\documentclass[12pt]{article}
\usepackage{amsmath,amssymb,amsthm}
\usepackage[a4paper,margin=1in]{geometry}
\usepackage{hyperref}
\usepackage{graphicx}
\usepackage{algorithm}
\usepackage{algpseudocode}
\usepackage{xcolor}
\usepackage{tikz}
\usepackage[most]{tcolorbox}
\usetikzlibrary{matrix, arrows.meta}

\tcbset{
  stepbox/.style={
    enhanced,
    colback=white,
    colframe=blue!50!black,
    fonttitle=\bfseries,
    coltitle=black,
    boxrule=0.8pt,
    arc=2mm,
    left=1mm, right=1mm, top=1mm, bottom=1mm,
    title={Step \thetcbcounter},
    before skip=10pt, after skip=10pt,
  }
}

\title{{On the most reliable graphs with fixed redundancy}}

\author{
Rotem Brand \and
Reuven Cohen \and
Simi Haber \and
Baruch Barzel
}

\date{}

\begin{document}

\maketitle

\begin{center}
\small
Department of Mathematics, Bar-Ilan University, Ramat Gan, Israel\\
\vspace{2mm}
*Corresponding author(s). E-mail(s): \texttt{rotembrand1234@gmail.com, reuven@math.biu.ac.il, simi@math.biu.ac.il, baruchbarzel@gmail.com.}
\end{center}

\newcommand{\cs}{\ensuremath{\mathcal{CS}}}    
\newcommand{\mcs}{\ensuremath{\mathcal{C}}}
\newcommand{\tmcs}{\ensuremath{\mathcal{C}^t}}
\newcommand{\ntmcs}{\ensuremath{\mathcal{C}^n}}

\theoremstyle{plain}
\newtheorem{theorem}{Theorem}[section]
\newtheorem{lemma}[theorem]{Lemma}
\newtheorem{corollary}[theorem]{Corollary}
\newtheorem{proposition}[theorem]{Proposition}
\newtheorem{fact}[theorem]{Fact}
\newtheorem{observation}[theorem]{Observation}
\newtheorem{claim}[theorem]{Claim}

\theoremstyle{definition}
\newtheorem{definition}[theorem]{Definition}
\newtheorem{example}[theorem]{Example}
\newtheorem{conjecture}[theorem]{Conjecture}
\newtheorem{open}[theorem]{Open Problem}
\newtheorem{problem}[theorem]{Problem}
\newtheorem{question}[theorem]{Question}

\theoremstyle{remark}
\newtheorem{remark}[theorem]{Remark}
\newtheorem{note}[theorem]{Note}

\begin{abstract}
The all-terminal reliability of a graph $G$ is the probability that $G$ remains connected when each edge fails independently with probability $p$. For fixed $n$ and $m$, the uniformly most reliable problem asks which graph with $n$ vertices and $m$ edges maximizes reliability for all $p \in [0,1]$. Although such graphs do not always exist, optimal graphs in the regime $p \to 0$ always do and are determined by the structure of their minimal cut sets.

We establish a structural characterization of graphs that are most reliable near $p=0$. Our results partially resolve a conjecture of Bourel et al., showing that, under suitable conditions, regular graphs with maximal girth are optimal. Extending this analysis to graphs with fixed redundancy $r=m-(n-1)$ and sufficiently large $n$, we show that the most reliable graphs are obtained by subdividing the most reliable cubic graphs with $2(r-1)$ vertices. The general conjecture remains open.

Unlike previous results, which resolved only small redundancy cases or very dense regimes, our approach yields a substantial extension of the known range. We determine the unique cubic candidates for uniformly most reliable graphs for all redundancy levels $m-n \le 19$, and prove the non-existence of uniformly most reliable graphs for several infinite families with fixed redundancy and asymptotically large $n$. These results significantly enlarge both the candidate class and the range of provable non-existence.
\end{abstract}

\section{Introduction}
The problem of identifying uniformly most reliable graphs isolates the purely combinatorial impact of topology on network reliability. We consider graphs with unreliable edges, where each edge independently fails with equal probability $p$. The \textit{all-terminal reliability} of a graph is the probability that all vertices remain connected under this failure model. A \textit{uniformly most reliable graph} with $n$ vertices and $m$ edges is a graph that maximizes reliability for all values of $p$. Understanding such graphs clarifies how structural properties alone determine reliability.

Although the model is simple, it has broad applicability across infrastructure networks, including communication systems and electrical distribution grids. Several examples of \emph{uniformly most reliable graphs} have been identified in the literature~\cite{boesch1991existence,wang1994proof,romero2017building,Rela2017PetersenGI,canale2019building}. These graphs are closely related to subdivisions of $3$-regular graphs, in which each edge of a cubic graph is replaced by a path~\cite{wang1997structure}.

Since any connected graph with $n$ vertices has at least $n-1$ edges, it is natural to define the \emph{redundancy} (also called the circuit rank) as $r = m - (n - 1)$, where $m$ is the total number of edges.

Uniformly most reliable graphs do not always exist~\cite{myrvold1991uniformly}. Consequently, the notion of a \textit{most reliable graph near zero or one} is introduced~\cite{brown2014nonexistence}. A graph is most reliable near zero (respectively, near one) if it maximizes reliability for all $p$ in a sufficiently small interval near zero (respectively, near one), among graphs with the same numbers of vertices and edges. Unlike uniform optimality, optimality near zero or one always exists. Since infrastructure components are typically highly reliable, our focus is on graphs that are most reliable near zero.

This article develops a method for constructing graphs that are most reliable near zero. Our main structural result shows that regular graphs of maximal girth satisfying an explicit structural condition are most reliable near zero. In particular, this provides further evidence for the conjecture of Bourel et al~\cite{GRASPHeuristics} that uniformly most reliable regular graphs attain maximum girth. While we verify this conjectural principle under additional conditions, the general case remains open.

Extending this analysis to graphs with fixed redundancy $r$ and sufficiently large $n$, we prove that the most reliable graphs are obtained by subdividing the most reliable cubic graphs with $2(r-1)$ vertices. By enumerating $3$-regular graphs of large girth with a small number of vertices, we identify the only candidates for the most reliable cubic graphs on fewer than $38$ vertices. These serve as the basis for determining uniformly most reliable graphs with up to $20$ redundant edges and a large number of vertices (see Figure~\ref{fig:most_reliable_graphs}).

We also identify a structural property governing reliability near one and use it to prove the non-existence of uniformly most reliable graphs for several infinite families with constant redundancy and asymptotically large numbers of vertices.

The core contribution of our work is a unifying structural principle that governs reliability across a wide range of redundancy levels. Rather than focusing on isolated values of $r$, our approach reveals how girth and the arrangement of subdivided paths interact to optimize reliability as the number of vertices is sufficiently large. This structural insight enables the systematic identification of candidate graphs for uniform optimality and yields tight upper bounds on achievable reliability as a function of graph size and redundancy.
 
The remainder of the paper is organized as follows. Section~\ref{sec:overview} introduces the core framework and algorithmic strategy and illustrates the main ideas without detailed proofs. Section~\ref{sec:regular} develops the theoretical foundation, showing that regular graphs with maximal girth optimize reliability near $p=0$. Section~\ref{sec:sparse} leverages these results to characterize reliable graphs with constant redundancy and a large number of vertices. Finally, Section~\ref{sec:most_reliable} applies the proposed algorithm to a database of over $5{,}000$ cubic graphs and identifies candidates for uniform optimality.

\section{Outline of the Main Results}\label{sec:overview}

This section outlines the main ideas and methodologies of the article and provides the intuition behind our approach. Subsequent sections present detailed proofs and technical derivations. Here, we focus on introducing the core concepts and strategies that form the foundation of our results.

For consistency in notation, we assume that a graph $ G $ consists of $ n $ vertices and $ m $ edges. Note that any connected graph must contain at least $ n-1 $ edges. Accordingly, we define the number of redundant edges in the graph as $ r = m - (n-1) $, also called the cyclomatic number. This value can be interpreted as the "budget" for additional edges beyond the minimum required for connectivity. As is standard, we denote the set of vertices in the graph by $ V(G) $ and the set of edges by $ E(G) $.

\subsection{The most reliable graphs}
Each edge of $G$ fails independently with a constant \emph{edge failure probability} $p$. 
The \emph{all-terminal reliability} (ATR) of $G$, denoted by $R_G(p)$, is defined as the probability that the network remains connected. 
For convenience, throughout this paper, we work with the \emph{unreliability polynomial}
\begin{equation}
U_G(p) = \Pr(G \text{ is disconnected}).
\end{equation}

Each event in the probability space corresponds to exactly $ k $ failing edges and $ m-k $ working edges. The probability of such an event is $ p^k q^{m-k} $, where $ q = 1-p $ represents the edge working probability. Assuming there are $ b_k $ cut sets with exactly $ k $ failing edges, $ U_G(p) $ can be expressed in the Bernstein representation:
\begin{equation} \label{eq:binom}
U_G(p) = \sum_{k=1}^m b_k p^k q^{m-k}.
\end{equation}
This representation shows that $ U_G $ is an $ m $-degree polynomial in the variable $ p $. These polynomials are monotonically increasing and satisfy the boundary conditions $ U_G(0) = 0 $ and $ U_G(1) = 1 $. Additionally, $ U_G(p) $ can also be expressed in its standard form:
\begin{equation}
\label{eq:power_series}
U_G(p) = \sum_{k=1}^m a_k p^k.
\end{equation}

\begin{definition}[Uniformly Most Reliable Graph]
A graph with $n$ vertices and $m$ edges is called \emph{uniformly most reliable} if it minimizes the unreliability polynomial $U_G(p)$ among all graphs with the same number of vertices and edges, for every $p \in [0,1]$.
\end{definition}

While identifying uniformly most reliable graphs is the central objective of reliability optimization, it has been shown in 1991 that such graphs do not always exist~\cite{myrvold1991uniformly}. This limitation motivates the introduction of weaker notions of optimality that capture reliability in a localized regime~\cite{brown2014nonexistence}.

\begin{definition}[Most Reliable Graph Near Zero or One]
A graph with $n$ vertices and $m$ edges is called \emph{most reliable near zero} if there exists $p' > 0$ such that it minimizes the unreliability polynomial $U_G(p)$ among all such graphs for all $p \in [0, p']$. Similarly, it is called \emph{most reliable near one} if there exists $p'' < 1$ such that it minimizes $U_G(p)$ for all $p \in [p'', 1]$.
\end{definition}

Unlike uniformly most reliable graphs, a most reliable graph near zero or one always exists, as the number of graphs with fixed $n$ and $m$ is finite.

In this article, we primarily focus on the regime near $p = 0$, which corresponds to systems composed of highly reliable components, as is typical in many real-world infrastructure networks. Additionally, we note that any uniformly most reliable graph, if it exists, must necessarily be the most reliable near both zero and one.

Our primary tool for identifying the most reliable graphs is the \emph{coefficients comparison method}. This method evaluates the reliability of graphs by comparing their coefficient vectors, $ (b_1, b_2, \dots, b_m) $, derived from the Bernstein representation of the unreliability polynomial. To prove that a graph is uniformly most reliable, it is enough to show that each of its coefficients is minimal among all graphs with the same number of vertices and edges. 

For the case of most reliable graphs near zero, we define a total order for two graphs $ G $ and $ H $ such that $ G < H $ if the coefficient vector of $ G $ is lexicographically smaller than that of $ H $. In this order, the comparison begins with the first coefficient; if they are equal, the second coefficient is compared, and so on. Similarly, for most reliable graphs near one, we define an analogous total order using the lexicographical ordering of the reverse coefficient vectors. 

The coefficients comparison method relies on the following fundamental lemma:

\begin{lemma}[Coefficients Comparison \cite{brown2014nonexistence}]
\label{lem:coefficients_comparison}

Let $ f(p) = \sum_{k=1}^{m} c_k p^k (1-p)^{m-k} $ and $ g(p) = \sum_{k=1}^{m} d_k p^k (1-p)^{m-k} $ be two polynomials. The following holds:

\begin{enumerate}
    \item If there exists $ i $ such that $ \forall j \in [0, i): c_j = d_j $ and $ c_i \leq d_i $, then there exists $ p' $ such that $ \forall p \in [0, p'] : f(p) \leq g(p) $;
    
    \item If there exists $ i $ such that $ \forall j \in [0, i): c_{m-j} = d_{m-j} $ and $ c_{m-i} \leq d_{m-i} $, then there exists $ p' $ such that $ \forall p \in [p', 1] : f(p) \leq g(p) $;
    
    \item If $ \forall i : c_i \leq d_i $, then $ \forall p \in [0, 1] : f(p) \leq g(p) $;
\end{enumerate}
\end{lemma}

Using the coefficient comparison method, we can define a sequence of reliability classes, $ \{\mathcal{A}_k\}_{k=0}^m $. The initial class, $ \mathcal{A}_0 $, corresponds to $\mathcal{G}_{n,m} $, the set of all multigraphs with $ n $ vertices and $ m $ edges. Each subsequent class, $ \mathcal{A}_k $, consists of the graphs in $ \mathcal{A}_{k-1} $ that have the smallest $k$-coefficient among all graphs in the previous class. As we progress through the sequence, the classes become smaller, and the most reliable graph near zero can be found in $ \mathcal{A}_m $. Since the first $k$ coefficients in each representation of the unreliability polynomial uniquely determine the first $k$ coefficients in the other, the classes $\mathcal{A}_k$ may equivalently be defined using either the binomial coefficients~\eqref{eq:binom} or the power-series coefficients~\eqref{eq:power_series}. Figure~\ref{fig:reliability_classes} illustrates an example of the reliability classes for graphs with 16 vertices and 24 edges.

Similarly, we can construct a sequence of sets $ \{B_k\}_{k=0}^m $, where $ B_0 $ is also $ \mathcal{G}_{n,m} $. In this case, each subsequent set $ B_k $ contains the graphs in $ B_{k-1} $ that have the smallest $ (m-k) $-coefficient. The most reliable graph near one can be found in $ B_m $.

\subsection{Previous results}
\label{subsec:prev_res}
An important observation is that the most reliable graphs near zero exhibit maximal edge connectivity. Edge connectivity is defined as the minimum number of edges whose removal disconnects the graph. If the edge connectivity of a graph is $ k $, then all coefficients $ b_1 $ to $ b_{k-1} $ in its unreliability polynomial are zero, placing the graph in the $ \mathcal{A}_{k-1} $ class.

Harary \cite{harary1962maximum} constructed a family of graphs with maximal edge connectivity. A graph with maximal edge connectivity maximizes the minimum degree of its vertices, as the edge connectivity of a graph is at most equal to its minimum degree. For a graph with $ n $ vertices and $ m $ edges, the minimum degree is $ \lfloor \frac{2m}{n} \rfloor $, and for the graphs that maximize the edge connectivity, the maximum degree is $ \lceil \frac{2m}{n} \rceil $. In particular, if $n$ divides $2m$, then the graph with maximal edge connectivity is $ k $-regular, where $ k = \frac{2m}{n} $.

Near one, the reliability of a graph depends on the number of spanning trees in the graph. Using the coefficient compression method, for coefficients $ \{b_k\} $ in the range $ m - (n-2) \leq k \leq m $, we have $ b_k = \binom{m}{k} $, as connecting the graph requires at least $ n-1 $ working edges. However, $b_r$ for $r=m-(n-1)$ is equal to, $ b_{r} = \binom{m}{r} - T(G) $, where $ T(G) $ represents the number of spanning trees of the graph, also known as the \emph{tree number}. Consequently, the most reliable graph near one, with respect to the ATR polynomial, maximizes the tree number.

This research primarily focuses on sparse networks, where $n \leq m < 1.5n$. For the most reliable sparse graphs near zero, the minimum degree is two, which ensures that the edge connectivity is also maximized at two. In these graphs, a chain of $c-1$ consecutive vertices with a degree of two can be replaced by a single edge with a failure probability of $ 1 - q^{c} $. The resulting transformed graph is referred to as the \emph{structure graph} or \emph{distillation}, and it is denoted by $ S(G) $. In the structure graph, the edges are called chains, while the vertices are referred to as structure vertices. Figure~\ref{fig:structure} illustrates an example of a graph and its corresponding structure graph.

\begin{figure}
    \centering
    \includegraphics[width=0.7\linewidth]{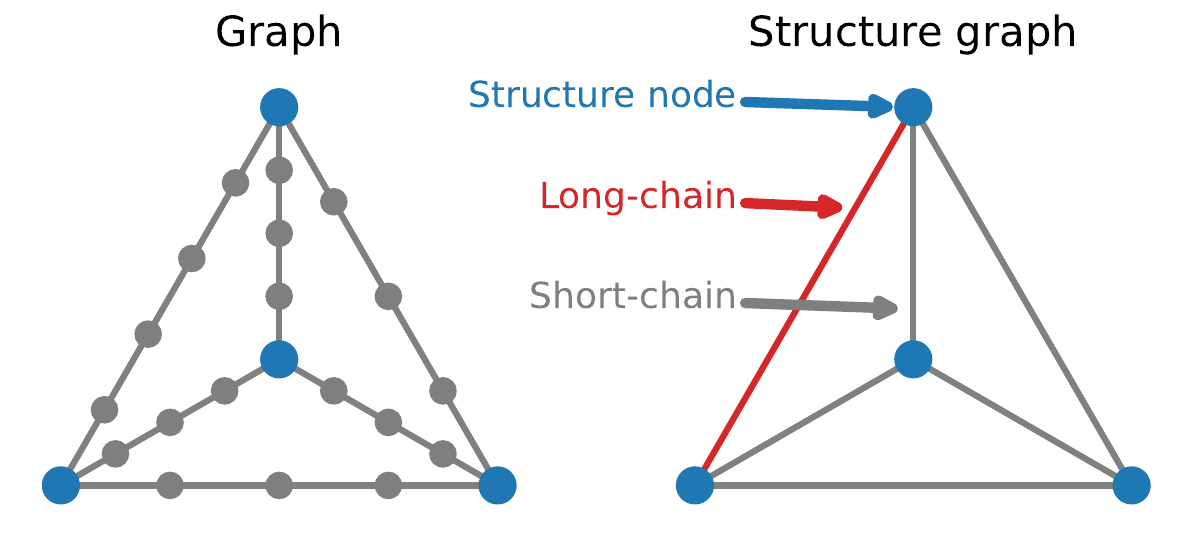}
    \caption{An example of a structure graph. Consecutive vertices with a degree of two form a chain, while vertices with a degree other than two become structure vertices. In optimal graphs, the lengths of the chains differ by at most one. In this scenario, we classify the chains as short or long.}
    \label{fig:structure}
\end{figure}

According to the findings in \cite{wang1997structure}, the most reliable graph near zero is a 3-connected, 3-regular structure graph with nearly equal chain lengths. To minimize the impact of cut sets with two edges, the structure graph must be 3-connected, ensuring that the only cut sets with two edges are pairs of edges within the same chain. Furthermore, reducing the number of such edge pairs requires minimizing chain lengths, which implies that the lengths of all chains should differ by at most one. Setting the structure graph to be 3-regular further maximizes the number of chains, distributing the connectivity more evenly.

In summary, a graph with a 3-connected, 3-regular structure graph, where the lengths of the chains differ by at most one, minimizes the second coefficient of the unreliability polynomial. The chains in such a graph can be classified as either short or long chains. An example of this structure is illustrated in Figure~\ref{fig:structure}.

Next, we focus on minimizing the third coefficient of the unreliability polynomial. It is important to note that each structure node in the structure graph inherently forms a trivial cut set with the three chains connected to it. To minimize the third coefficient, we require that the only cut sets involving three chains correspond to those formed by a single structure node. Furthermore, in the structure graph, the number of cut sets with three edges increases if two long chains connect to the same structure node. To mitigate this, it is essential to minimize the number of long chains intersecting at the same structure node \cite{wang1997structure}. Figure~\ref{fig:structure} illustrates an example of short and long chains intersecting at a single structure node.

\subsection{ Main Results}
In this section, we present the main methods and results used in this paper. We start in the case where $m=kn/2$ for natural $k\geq3$, resulting in an average degree of $k$. We show that in some cases, the most reliable graph near zero is a $k$-regular graph with high girth. This result is formally proved in Section~\ref{sec:regular}.

Then, we focus on the sparse case where $m<1.5n$ and the number of vertices is large. In this scenario, the average degree is less than 3. We demonstrate that, in this case, the uniformly most reliable graph must have a structure of the most reliable 3-regular graph near zero of the appropriate size, and we also discuss the optimal placement of the long chains. Finally, we find that in some cases, the most reliable graph near one has non-balanced chain lengths, i.e., chain lengths differ by more than one, rendering the uniformly most reliable graph non-existent. We develop those results deeply in Section~\ref{sec:sparse}.

\subsubsection{High-Girth Graphs Are Reliable Near Zero}

Our main result shows that, under appropriate conditions, regular graphs with high girth are the most reliable graphs near zero. We present the main ideas behind this result in this section and provide a formal proof in Section~\ref{sec:regular}.

This observation is especially relevant for the design of sparse graphs with a fixed number of redundant edges. In Section~\ref{sec:sparse}, we show that in the case of graphs with $r$ redundant edges and a large number of vertices, the uniformly most reliable graph has a structure graph that is the most reliable 3-regular graph near zero with $2(r-1)$ nodes. Therefore, each identification of the most reliable 3-regular graph near zero corresponds to an infinite family of uniformly most reliable graphs (if they exist).

To understand why low-girth graphs are less reliable, we need to examine their minimal cut sets. A \textit{minimal cut set} is a cut set that does not contain any smaller cut set, effectively representing the "atomic elements" of cut sets. Each minimal cut set divides the graph into exactly two connected components. By focusing on the smaller component, we observe that each minimal cut set is induced by a connected subgraph.

Cycles in the graph create low-order cut sets, reducing the reliability of low-girth graphs. For instance, in the most reliable 3-regular graph near zero, the only cut sets with three edges are those induced by a single vertex. In contrast, Figure~\ref{fig:cut_sets} illustrates how a cycle with three vertices induces an additional cut set with three edges. Consequently, a graph with a girth of three cannot be the most reliable near zero, as determined by the coefficient comparison method. Proposition~\ref{prop:The_size_of_the_induced_cut_set} extends this result, demonstrating that cycles induce lower-order cut sets than trees.

\begin{figure}
    \centering
    \includegraphics[width=1\linewidth]{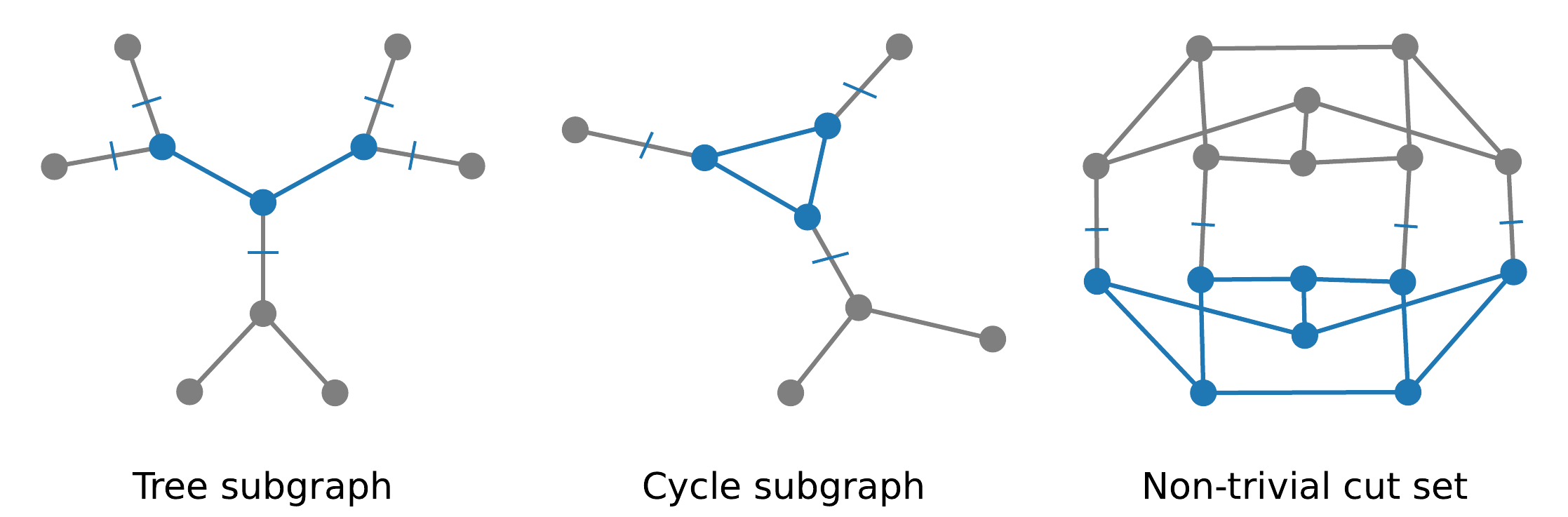}
\caption{Various Subgraphs(Skeletons) and their induced cut sets.}
\label{fig:cut_sets}
\end{figure}

\begin{proposition}[The Size of the Induced Cut Set]
\label{prop:The_size_of_the_induced_cut_set}
Let $ G $ be a $ k $-regular graph. For any subgraph $ X $, denote the number of connected components, edges, and vertices in $ X $ by $ c(X) $, $ m(X) $, and $ n(X) $, respectively. Let $\partial(X)$  represent the induced cut set. Then:
\begin{enumerate}
    \item $ |\partial(X)| = k \cdot n(X) - 2 \cdot m(X) $;
    \item In particular, if $ X $ is a forest, then $ |\partial(X)| = n(X) \cdot (k-2) + 2 \cdot c(X) $;
    \item If $ X $ is a cycle, then $ |\partial(X)| = n(X) \cdot (k-2) $.
\end{enumerate}
\end{proposition}

The second and third propositions follow directly from the first proposition, which is easily justified by noting that each vertex in $ X $ connects to $ k $ edges, while each internal edge reduces the number of external edges by two. Figure~\ref{fig:cut_sets} illustrates several subgraphs and their corresponding induced cut sets. From this, we conclude that cycles generate smaller cut sets than trees, and in both cases, the size of the cut set is determined by the number of vertices in the cycle or the tree. Consequently, to minimize the presence of low-order cut sets, we aim to avoid short cycles.

In this context, it is convenient to relax the notion of \textit{trivial cut sets} as follows. In a graph with a girth of at least $ g $, the neighborhood of any vertex containing at most $ g-1 $ vertices is isomorphic to a subgraph of the infinite $ k $-regular tree. Consequently, the neighborhoods of vertices from different graphs with a girth of at least $ g $ are isomorphic (Theorem~\ref{lem:Global-homomorphism3}). Each subgraph with $ \ell \leq g-1 $ vertices induces a cut set with $ \ell(k-2)+2 $ edges. These cut sets are considered "trivial" because they appear in all graphs with sufficiently high girth.

If a graph contains only trivial cut sets up to $ \ell(k-2)+2 $ edges, it is more reliable near zero compared to graphs with non-trivial cut sets. Specifically, if a graph is free of non-trivial cut sets up to $ g $ edges, this suggests that the optimal graph near zero has a girth of at least $ g $. Figure~\ref{fig:reliability_classes} illustrates an example of reliability classes in $ \mathcal{G}_{16,24} $. We observe that in this figure, graphs in $ \mathcal{A}_k $ have a girth of at least $ k-2 $. Moreover, one of the graphs in $ \mathcal{A}_3 \setminus \mathcal{A}_4 $ has a girth of 4 but also contains a non-trivial cut set, providing an example of a non-trivial cut set that is not induced by a cycle.

\begin{figure}
    \centering
    \includegraphics[width=1\linewidth]{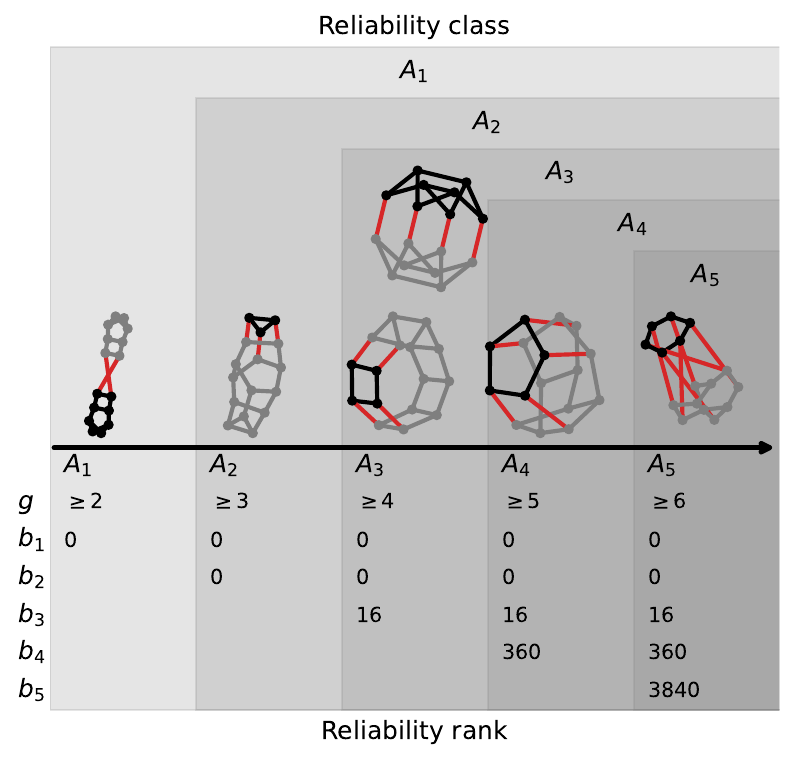}
    \caption{The reliability classes for graphs with $16$ vertices and $24$ edges. All the graphs that are not in the class $\mathcal{A}_5$ contain non-trivial cut sets(\textcolor{red}{\textbf{red}}) that are induced from a short cycle or from a bigger subgraph(\textbf{black}).}
    \label{fig:reliability_classes}
\end{figure}

We now present the main theorem of this article, which forms the foundation of our method for constructing highly reliable networks using high-girth graphs. The proof of this theorem is provided in Subsection~\ref{subsection:main_theorem}.

\begin{theorem}[Regular Graphs Without Non-Trivial Cut Sets] \label{thm:Most-reliable-graphs-regular}
Let $ n \geq 2 $ such that $ m=kn/2 \in \mathbb{N} $. Define $\mathcal{D}_{n,m}(g) $ as the set of all $ k $-regular graphs in $ \mathcal{G}_{n,m} $ that do not contain non-trivial cut sets with at most $ g(k-2) $ edges. Suppose there exists $ g \geq 4 $ such that $ \mathcal{D}_{n,m}(g) \neq \emptyset $. Then:
\[
\mathcal{A}_{g(k-2)} = \mathcal{D}_{n,m}(g).
\]
\end{theorem}

By generating regular graphs with high girth and verifying the absence of non-trivial cut sets, we can significantly narrow the search space for the most reliable graphs near zero. Techniques from \cite{meringer1999fast}, \cite{mckay1998fast}, and \cite{brinkmann1995smallest} enable the systematic generation of all regular graphs with a specified girth. Additionally, a comprehensive collection of cubic graphs is available on the House of Graphs website\footnote{https://houseofgraphs.org/meta-directory/cubic} \cite{coolsaet2023house}. Although the number of graphs with the highest girth becomes substantial for larger graphs, it remains relatively small for smaller graphs.

Figure~\ref{fig:most_reliable_graphs} depicts the most reliable graphs with fewer than 20 redundant edges, which includes numerous graphs with maximal girth. After verifying the absence of non-trivial cut sets in these graphs, we conclude that they exhibit superior reliability near zero.

Furthermore, we assert that these graphs are the sole candidates for uniformly most reliable graphs, representing significant progress in their identification. Thus far, only graphs with fewer than seven redundant edges have been confirmed as uniformly most reliable \cite{boesch1991existence,wang1994proof,romero2017building,Rela2017PetersenGI,canale2019building}. Additionally, it has been postulated in \cite{bourel2019building} that the Cantor–Möbius graph also qualifies as a uniformly most reliable graph.

\begin{figure}
    \centering
    \includegraphics[width=0.75\linewidth]{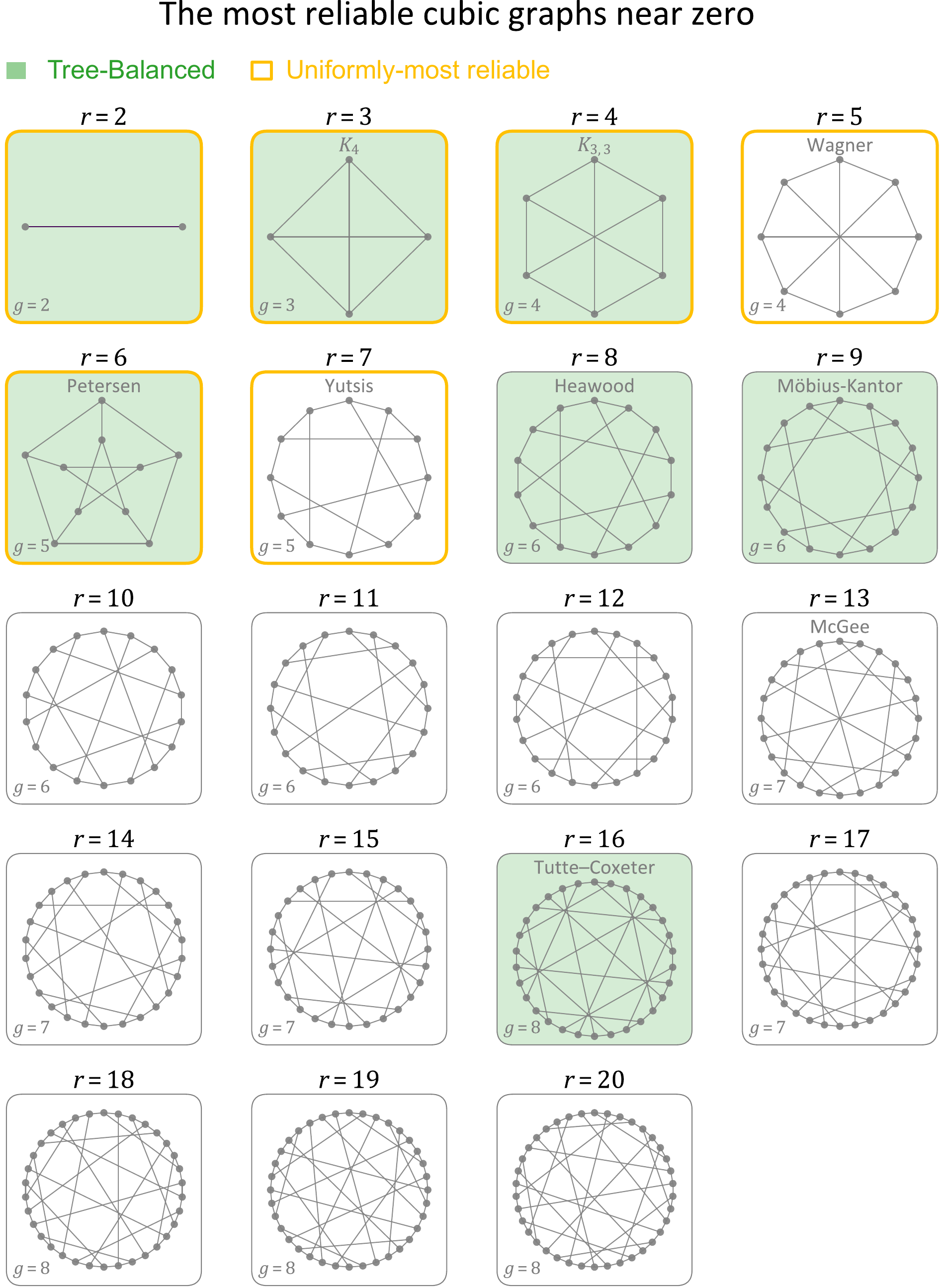}
\caption{Most reliable graphs near zero with a given redundancy $r$, and number of nodes $n=2(r-1)$. $g$ denotes the girth of the graph. Green graphs are Tree-Balanced, and gold graphs are previously proven to be uniformly most reliable.}
\label{fig:most_reliable_graphs}
\end{figure}

\subsubsection{Reliable sparse graphs}
\label{subsec:sparse}

In this subsection, we focus on the case of sparse, reliable graphs. A graph is called \emph{sparse} if it has $n$ vertices and $m$ edges satisfying $m < 1.5n$. This threshold is motivated by the fact that 3-regular graphs require exactly $1.5n$ edges. Thus, sparse graphs lie below the minimal edge count needed for 3-regularity, and therefore have edge-connectivity at most 2.

As mentioned previously, existing results suggest that graphs minimizing $b_2$ have a \textit{structure graph} that is both 3-edge-connected and 3-regular, with chain lengths that differ by at most one \cite{wang1997structure,landgren2024newframeworkidentifyingreliable}.

Constructing a reliable sparse graph with non-uniform chain lengths involves two key stages: identifying a reliable structure graph and determining the placement of the long chains within it. Placing two or more long chains within the same cut set of the structure graph increases the probability of failure for that cut. Since the low-order minimal cut sets in reliable cubic graphs typically consist of closely connected edges, our objective is to maximize the pairwise distances between the long chains in the structure graph.

Nevertheless, regular graphs with low girth often exhibit higher diameter, as the neighborhood of each edge contains relatively few edges. This leads to a trade-off between two competing requirements: using a structure graph with high girth, and maintaining sufficient distance between long chains. 

Fortunately, when the chains are long enough, the influence of their placement becomes negligible. Consequently, for a constant $r$ and a large value of $n$, the uniformly most reliable sparse graph must have the most reliable structure graph near zero. Therefore, methods for constructing reliable cubic graphs can be applied to produce highly reliable sparse graphs with long chains.

\subsubsection*{Reliability decomposition}

We begin by decomposing the unreliability polynomial to express the reliability of the sparse graph using the combinatorial language of its structure graph.

The number of long chains in the graph depends on the parameters $n$ and $r$. The redundancy of the structure graph is equal to that of the original graph. In the corresponding cubic-structure graph, each structure vertex is incident to $1.5$ times the number of chains. Therefore, the structure graph consists of $2(r - 1)$ structure vertices and $3(r - 1)$ chains. Assuming the chains are distributed evenly, the remaining $n - 2(r - 1)$ vertices are divided among the $3(r - 1)$ chains.

Let $c, \lambda \in \mathbb{N}$ be defined by
\begin{equation}\label{eq:c_lambda}
n - 2(r - 1) = (c - 1) \cdot 3(r - 1) + \lambda,
\end{equation}
with $0 \leq \lambda < 3(r - 1)$. Here, $c \geq 1$ corresponds to the number of edges in each short chain, and $\lambda$ denotes the number of long chains.

We classify failure events into two types: (1) Failure events in which each chain contains at most one failing edge, (2) Failure events in which at least one chain contains two or more failing edges. As these two event types are disjoint, we decompose the total unreliability as:
\begin{equation} \label{eq:ATR_composed}
U_G = U_G^{(1)} + U_G^{(2)},
\end{equation}
where $U_G^{(i)}$ denotes the probability of a type $(i)$ failure event.

The contribution $U_G^{(2)}$ depends only on the chain lengths and not on the structure graph. Specifically, the probability that a chain $e$ of length $c_e$ has at most one failing edge is given by $q^{c_e} + c_e p q^{c_e - 1}$. Therefore, if $S(G)$ denotes the structure graph of $G$, then:
\begin{equation}
U_G^{(2)} = 1 - \prod_{e \in E(S(G))} \left(q^{c_e} + c_e p q^{c_e - 1}\right).
\end{equation}

In contrast, $U_G^{(1)}$ depends on both the structure graph and the placement of long chains. Let $b_k^{(1)}$ be the number of type (1) $k$-cut sets in $S(G)$. Then:
\begin{equation}
\label{eq:U_G1}
U_G^{(1)} = \sum_{k = 0}^{3(r - 1)} b_k^{(1)} p^k (1 - p)^{m - k}.
\end{equation}

Since all sparse graphs in $\mathcal{A}_2$ have cubic structure with the same number of chains and have chain lengths differing by at most one, they yield the same value for $U_G^{(2)}$. Hence, minimizing total unreliability reduces to minimizing $U_G^{(1)}$.

To further analyze $U_G^{(1)}$, note that each type (1) $k$-cut set in the sparse graph arises from a $k$-cut set $X$ of the structure graph. Suppose $X$ contains exactly $\ell$ long chains. Then $X$ induces
\begin{equation}
\label{eq:gamma_X}
\gamma^{(X)} = c^{k - \ell}(c + 1)^{\ell}
\end{equation}
distinct type (1) cut sets in the sparse graph, since we need to choose exactly one failing edge in each chain of $X$. Let $\Gamma_k$ denote the set of all $k$-cut sets in the structure graph. Expanding~\eqref{eq:gamma_X} above using the binomial theorem gives:
\begin{equation}
\begin{split}
\gamma^{(X)} &= \sum_{j = 0}^{\ell} \binom{\ell}{j}\, c^{k - j}\\
b_k^{(1)} &= \sum_{X \in \Gamma_k} \gamma^{(X)}
\end{split}
\end{equation}

Define $\Gamma_{k,\ell}$ as the subset of $\Gamma_k$ containing cut sets with exactly $\ell$ long chains. Grouping terms by the number of long chains selected yields:
\begin{equation}
\begin{split}
b_k^{(1)} &= \sum_{j = 0}^{k} \gamma_{k,j} \, c^{k - j}\\
\gamma_{k,j} &= \sum_{\ell = j}^{k} |\Gamma_{k,\ell}| \, \binom{\ell}{j}
\end{split}
\end{equation}
In other words, $\gamma_{k,j}$ counts the number of pairs $(X, L)$ where $X \in \Gamma_k$ is a $k$-cut set of the structure graph and $L \subseteq X$ is a subset of $j$ long chains. In particular, $\gamma_{k,0} = b_k(S(G))$, the number of $k$-cut sets in the structure graph $S(G)$.

Substituting into~\eqref{eq:U_G1}, we obtain:
\begin{equation}
\label{eq:U_G1_formula}
U_G^{(1)} = \sum_{k = 0}^{3(r - 1)} \sum_{j = 0}^{k} \gamma_{k,j} \cdot c^{k - j} p^k (1 - p)^{m - k}.
\end{equation}

\subsubsection*{Extension of the Coefficient Comparison Lemma}

When $c$ is large, the terms with smaller values of $j$ contribute more significantly to $b_k^{(1)}$, due to the dominance of higher powers of $c$. Therefore, minimizing $b_k^{(1)}$ in this regime requires a lexicographic minimization of the sequence $(\gamma_{k,0}, \gamma_{k,1}, \ldots, \gamma_{k,k})$. This observation naturally leads to an extension of the coefficient compression lemma.

\begin{lemma}[Extension of Coefficient Comparison]
\label{lem:Extension of coefficients comparison}
For each $r$, there exists $n_0$ such that for all $n > n_0$, the uniformly most reliable sparse graph near zero with $n$ vertices and $r$ redundant edges minimizes the coefficients $\gamma_{k,j}$ in lexicographic order over the pairs $(k,j)$. That is, the sequence
\[
\begin{tikzpicture}[every node/.style={anchor=center}]
  \matrix (m) [matrix of math nodes, row sep=0.4em, column sep=2em] {
    \gamma_{1,0} & \gamma_{1,1} &           &           \\
    \gamma_{2,0} & \gamma_{2,1} & \gamma_{2,2} &         \\
    \gamma_{3,0} & \gamma_{3,1} & \gamma_{3,2} & \gamma_{3,3} \\
    \dots     \\
  };

  \draw[-{Stealth}] (m-1-1) -- (m-1-2);

  \draw[-{Stealth}] (m-1-2) -- (m-2-1);

  \draw[-{Stealth}] (m-2-1) -- (m-2-2);
  \draw[-{Stealth}] (m-2-2) -- (m-2-3);

  \draw[-{Stealth}] (m-2-3) -- (m-3-1);

  \draw[-{Stealth}] (m-3-1) -- (m-3-2);
  \draw[-{Stealth}] (m-3-2) -- (m-3-3);
  \draw[-{Stealth}] (m-3-3) -- (m-3-4);
\end{tikzpicture}
\]
is minimized entry by entry in the specified traversal order.
\end{lemma}

\begin{proof}
For each $k$, we can write $b_k^{(1)} = \sum_{j = 0}^{k} \gamma_{k,j} c^{k - j}$. When $c$ is sufficiently large, the dominant contribution to $b_k^{(1)}$ comes from the leading coefficients in the lexicographic sequence $(\gamma_{k,0}, \gamma_{k,1}, \dots, \gamma_{k,k})$. Therefore, minimizing $b_k^{(1)}$ for large $c$ is equivalent to lex-minimizing this sequence.

According to the Coefficient Comparison Lemma~\ref{lem:coefficients_comparison}, minimizing $U_G^{(1)}$ near zero requires minimizing the sequence $(b_1^{(1)}, b_2^{(1)}, \dots, b_{3(r - 1)}^{(1)})$ in lexicographic order. It follows that minimizing $U_G^{(1)}$ near zero requires minimizing the full sequence of $\gamma_{k,j}$ coefficients in the lexicographic order over $(k, j)$.
\end{proof}

The coefficients $\gamma_{k,j}$ depend only on the structure graph and the location of long chains; they are independent of the actual chain lengths. It follows that, for any fixed redundancy level and number of long chains, there exists a combination of structure graph and long chain placement that yields the most reliable graph near zero for all large values of $n$.

\begin{definition}[Marked structure]
A \emph{marked structure graph} is a pair consisting of a cubic graph $G$ and a subset $L \subseteq E(G)$ of edges designated as long chains.
\end{definition}

\begin{definition}[Asymptotically most reliable marked structure]
A marked cubic structure graph with $r$ redundant edges and $\lambda$ long chains is called \emph{asymptotically most reliable} if it minimizes the coefficients $\gamma_{k,j}$ in lexicographic order over $(k,j)$ among all cubic marked structures with the same redundancy and number of long chains.
\end{definition}

\begin{theorem}[Most reliable large graphs near zero]
For each $r \geq 2$, there exists $n_0(r)$ such that, for all $n > n_0$, the most reliable graph near zero with $n$ vertices and $r$ redundant edges has the asymptotically most reliable marked structure with $r$ redundant edges and $\lambda$ long chains, where $\lambda$ satisfies
\[
n = 2(r-1) + (c-1)\cdot 3(r-1) + \lambda,
\]
for $0 \leq \lambda < 3(r-1)$.
\end{theorem}

This theorem follows directly from Lemma~\ref{lem:Extension of coefficients comparison} and provides the main tool for identifying the most reliable graphs with constant redundancy and a large number of vertices.

We extend the reliability classes $\{\mathcal{A}_k\}$ to a two-parameter family $\{\mathcal{A}_{k,j}\}$, defined for $0 \leq k \leq 3(r - 1)$ and $0 \leq j \leq k$. The class $\mathcal{A}_{k,j}$ consists of all marked structure graphs that minimize $\gamma_{k,j}$ among those in $\mathcal{A}_{k,j-1}$ (for $j > 0$), or among those in $\mathcal{A}_{k-1,k-1}$ when $j = 0$. 

This recursive construction ensures that each class $\mathcal{A}_{k,j}$ contains all graphs that are lexicographically minimal in $\gamma_{k',j'}$ for all $(k',j') < (k,j)$ in lexicographic order. In particular, the asymptotically most reliable marked structure near zero lies in $\mathcal{A}_{3(r - 1),\,3(r - 1)}$.

Next, we show that when $r$ is fixed and $n$ is sufficiently large, there exist $p$ that the most reliable graph for this $p$ has a marked structure that also minimizes the coefficients $\{\gamma_{k,j}\}$ sequentially in \emph{transposed lexicographic order} on the index pairs $(j,k)$.

To see this, define $s = \frac{p}{1 - p} c$ and assume that $s$ is a small constant. Additionally, let $c = s^{-t}$ for some fixed $t > 1$, so that $c$ grows faster then $s^{-1}$. Under this substitution, the unreliability polynomial becomes
\begin{equation}
    U_G^{(1)} = \left(\frac{s^{-t}}{s^{-t}+s}\right)^m\cdot  \sum_{k = 0}^{3(r - 1)} \sum_{j = 0}^{k} \gamma_{k,j} \, s^{k + jt}.
\end{equation}

In Lemma~\ref{lem:transpose_coeff_comp} below, we show that for sufficiently large $t$, and small $s$, minimizing $U_G^{(1)}$ at the point $p = \frac{s}{c + s}$ corresponds to minimizing the coefficients $\gamma_{k,j}$ in transposed lexicographic order on $(j,k)$, that is:
\[
\gamma_{1,0}, \, \gamma_{2,0}, \, \dots, \, \gamma_{1,1}, \, \gamma_{2,1}, \, \dots, \, \gamma_{1,2}, \, \gamma_{2,2}, \, \dots
\]

Minimizing terms of the form $\gamma_{k,0}=b_k(S(G))$ is equivalent to reducing the $k$-th coefficient of the structure graph. Therefore, under the transpose lexicographic minimization over $(j,k)$, the uniformly most reliable graph, when $r$ is fixed, and $n$ is large, must have a structure graph that minimizes all $b_k$ sequentially, and hence corresponds to the most reliable cubic graph with $2(r - 1)$ vertices. This observation justifies the search for the most reliable cubic graphs near zero.

To minimize $\gamma_{k,1}$, we observe that if the structure graph contains only trivial cut sets of size at most $k$, then, by Lemma~\ref{lem:cut_set_contain_two_edges}, which we prove later, each chain appears in the same number of such cut sets. In this case, the value of $\gamma_{k,1}$ is independent of the placement of the long chains.

Finally, to minimize $\gamma_{k,2}$, we observe that each trivial cut set involves edges that lie within a distance of at most $ k-2$. In Lemma~\ref{lem:cut_set_contain_two_edges}, which will be proved later, we also show that every pair of chains at distance less than $k-2$ appears in the same number of trivial cut sets of size $k$.

As a result, minimizing $\gamma_{k,2}$ sequentially for $k$ is equivalent to minimizing the number of pairs of long chains whose distance in the structure graph is exactly $k - 2$.

However, for a given redundancy and a large number of vertices, several of the structural properties associated with uniformly most reliable graphs can be mutually conflicting. For instance, in graphs with high girth, each chain has more edges within its $d$-neighborhood, which increases the number of nearby chain pairs and makes it more difficult to maintain large distances between long chains.

In Subsection~\ref{subsec:unequal_chains}, we present additional properties of uniformly most reliable graphs for a large number of vertices. In some cases, these properties are incompatible, and we use this to argue that a uniformly most reliable graph cannot exist. In other cases, the properties are consistent and help reduce the search space for candidate graphs.

Finally, in Subsection~\ref{subsec:algorithm}, we summarize these results and present an algorithm for identifying the uniformly most reliable graph in the asymptotic regime. We then use the database of cubic graphs described in Subsection~\ref{subsec:cubic_db} to narrow the search space for asymptotically uniformly most reliable graphs in certain cases, and to rule out their existence in others.

\subsubsection{The Non-Existence of the Uniformly Most Reliable Graph}
We now shift our focus to the behavior of sparse graphs near $p = 1$. In this regime, reliability is determined by the number of spanning trees, also known as the \emph{tree number}. We present a method for maximizing the tree number in sparse graphs with long chains and prove its correctness in Subsection~\ref{subsec:maximize_the_tree_number}.

In that analysis, the number of edges in a chain $e_i$ is expressed as $c + r_i$, where $r_i$ may vary (including negative values). It is shown that the tree number can be expressed as a polynomial in the variable $c$, and as $c$ increases, the objective becomes to minimize the polynomial’s coefficients from the highest degree to the lowest.

This leads to two key insights. First, for sufficiently large $c$, the sparse graph that maximizes the tree number must also have a structure graph that maximizes the tree number. Second, there exist cases in which unbalanced chain lengths (i.e., $r_i\notin\{0,1\}$) yield a higher tree number when the number of vertices is large.

We then evaluate each of the most reliable 3-regular graphs near zero to assess whether they can be the structure of a uniformly most reliable sparse graph. This analysis reveals several cases in which a uniformly most reliable graph does not exist, yielding the strongest known contradictions in this setting.

\begin{definition}[Tree-Balanced Graphs]
Let $G$ be a graph, and let $T(G)$ denote its tree number. We say that $G$ is \emph{tree-balanced} if for every pair of edges $e_1, e_2 \in E(G)$, it holds that $T(G - e_1) = T(G - e_2)$. 

Here, $T(G - e)$ denotes the tree number of the graph $G$ with edge $e$ removed.
\end{definition}

If the structure graph is not tree-balanced, then for sufficiently large $n$, one can increase the overall tree number by lengthening the chains corresponding to edges with higher $T(G - e)$ values and shortening those with lower values. As a result, the corresponding graph achieves a higher tree number than the one with uniform chain lengths.

Lemma~\ref{lem:non_UMRG_not_t_balanced} formalizes this: it asserts that non-tree-balanced structure graphs cannot correspond to uniformly most reliable graphs when the number of vertices is sufficiently large.

Figure~\ref{fig:compare_matching} illustrates this phenomenon using a graph in $\mathcal{G}_{8,12}$. We compare the most reliable graph near zero with a graph whose chain lengths differ by more than one but achieves a higher tree number.

Finally, Figure~\ref{fig:most_reliable_graphs} summarizes the cases where a uniformly most reliable graph does not exist. For each graph that is most reliable near zero, we verify whether it is tree-balanced.

\begin{figure}
    \centering
    \includegraphics[width=1\linewidth]{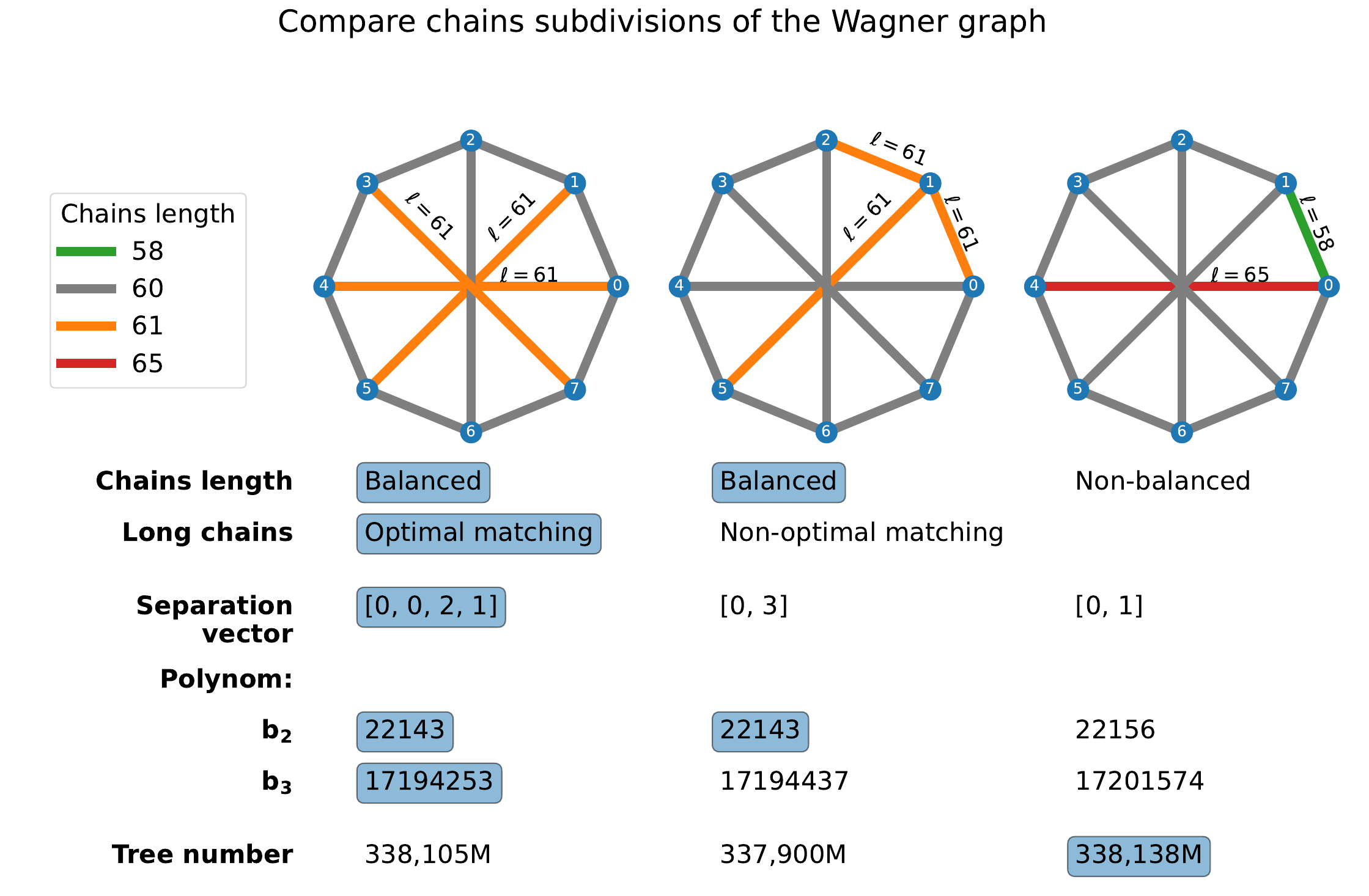}
    \caption{Comparison of different chain configurations on the Wagner graph, which is the most reliable graph near zero with $ 8 $ vertices and $ 12 $ edges. In the first and second graphs, the chain lengths differ by at most one. The first graph features long chains forming a matching with the optimal separation vector, making it the more reliable graph near zero. In contrast, the long chains in the second graph do not form a matching. In the third graph, all chains are of equal length except for one very long chain, which maximizes the tree number when removed, and one very short chain, which minimizes the tree number when removed. As predicted, the tree number of the third graph is the highest among the presented graphs.}

    \label{fig:compare_matching}
\end{figure}

\section{Reliable regular graphs}
\label{sec:regular}
This section rigorously establishes the paper's main result that regular graphs with high girth exhibit superior reliability near zero. In a $ k $-regular graph with high girth, the neighborhood of each vertex forms a $ k $-regular tree structure (see Figure~\ref{fig:embedding}), providing numerous alternative paths. For graphs with a girth of at least $ g $, it is possible to establish an isomorphism among connected subgraphs containing fewer than $ g $ vertices. By selecting the edges that connect these subgraphs to their complements, we define what are referred to as "trivial" cut sets. These cut sets are considered trivial in graphs with a girth of at least $ g $. Conversely, a cycle with $ g $ vertices induces another type of cut set, referred to as a "non-trivial" cut set. Therefore, a graph that is free of non-trivial cut sets, particularly one with higher girth, demonstrates improved reliability near zero. Some of our methods are based on expansion arguments. We chose to use direct terminology to keep the paper accessible beyond combinatorics.

To prove the main result, we establish an isomorphism between collections of trivial cut sets. We denote the collection of all minimal cut sets in $G$ by $\mcs_G$. The event ``all the edges of an edge set $X$ failed'' by $F_X$. One way of computing the unreliability of a graph is by using the inclusion-exclusion principle:
\begin{equation} \label{eq:inc_exc_principle_ATR}
    U_G = \Pr\left(\bigcup_{X \in \mcs_G} F_X\right) = \sum_{\emptyset\neq \mathcal{A} \subseteq \mcs_G}(-1)^{|\mathcal{A}|-1} \cdot p^{|\cup_{X \in \mathcal{A}} X|},
\end{equation}
Thus, if we can demonstrate an isomorphism between collections of trivial cut sets, we can conclude that a graph containing non-trivial cut sets is less reliable.

To establish this isomorphism, we proceed in three steps. First, we construct a local isomorphism between the neighborhoods of vertices. Next, we extend this isomorphism to encompass all trivial cut sets in the graphs. Finally, we demonstrate the isomorphism between the collections of trivial cut sets.

For convenience, we designate one of the vertices as a \emph{source vertex} and use it as a reference. Since this vertex can be chosen arbitrarily, the results remain valid for any selection of a source.

\subsection{Trivial Cut Sets}

In this subsection, we aim to establish both local and global isomorphisms. Let $ \mcs_G(\ell) $ denote the set of all minimal cut sets containing at most $ \ell $ edges. Since each minimal cut set $ X $ partitions the graph into exactly two subgraphs, we define the \textit{skeleton} of the cut set as the smaller connected component resulting from the failure of $ X $. If both resulting subgraphs are of equal size, the skeleton is taken to be the connected component containing the source vertex.

Conversely, for a connected subgraph $ X $ in $ G $, we define $ \partial(X) $ as the set of edges that connect $ X $ to its complement $ G - X $. Note that while $ \partial(X) $ always forms a cut set, it does not necessarily constitute a minimal cut set.

Figure~\ref{fig:cut_sets} illustrates minimal cut sets and their corresponding skeletons. As stated in Proposition~\ref{prop:The_size_of_the_induced_cut_set}, the size of the induced cut set is determined by the number of vertices and edges in the subgraph.

We define the \emph{trivial cut sets} of a graph $G$ with a girth $g$ as all the minimal cut sets induced from a tree with at most $(g-1)$ vertices. We denote the set of all trivial cut sets in $G$  as $\tmcs_G$ and all trivial cut sets with at most $\ell$  edges as $\tmcs_G(\ell)$. We also define the set of all the connected subgraphs in $G$ with at most $\ell$ vertices as $\sigma_G(\ell)$  and the set of all those subgraphs that intersect a vertex $v$ as $\sigma_G(\ell;v)$.

Because small subtrees uniquely determine trivial cut sets, our goal is to show that collections of small subtrees are isomorphic in graphs with the same girth. In this section, the comparison is carried out relative to a distinguished source vertex, which serves as a reference point for the isomorphism. Consequently, we work with rooted graphs. A \emph{rooted graph} is a graph together with a distinguished vertex, called the root. An isomorphism between rooted graphs is required to map the root to the root, thereby preserving the rooted structure. In particular, two rooted subgraphs are considered isomorphic only if the position of the root within the subgraph is preserved.
This rooted framework is used only in the present sparse-graph analysis; in later sections, the corresponding isomorphisms are established without reference to a distinguished vertex.

\begin{lemma}[Local Isomorphism]
   \label{prop:homomorphism-local}
    
   Let $ G_1 $ and $ G_2 $ be two $ k $-regular graphs with the same number of vertices and edges and a girth of at least $ \ell $. Let $ v_1 \in G_1 $ and $ v_2 \in G_2 $ be two vertices. Then, there exists a rooted isomorphism between 
   $ \sigma_{G_1}(\ell-1; v_1) $ and $ \sigma_{G_2}(\ell-1; v_2) $.
\end{lemma}

\begin{proof}
   We establish the isomorphism by performing a depth-first search (DFS) on each graph, starting from $ v_1 $ in $ G_1 $ and $ v_2 $ in $ G_2 $, and stopping after visiting $ \ell-1 $ vertices.

   Since the girth of both graphs is at least $ \ell $, each vertex in the DFS traversal appears at most once; otherwise, a cycle of fewer than $ \ell $ vertices would exist, contradicting the girth condition. 

   Suppose, for contradiction, that a subgraph is represented by two distinct DFS sequences. This would imply the presence of a cycle with fewer than $ \ell $ vertices, again contradicting the girth condition. Thus, each subgraph in $ \sigma_{G_1}(\ell-1; v_1) $ and $ \sigma_{G_2}(\ell-1; v_2) $ has a unique representation through the DFS process.

   Finally, we construct the required isomorphism by pairing subgraphs in $ \sigma_{G_1}(\ell-1; v_1) $ and $ \sigma_{G_2}(\ell-1; v_2) $ based on identical DFS sequences.
\end{proof}

Building on the local isomorphism established in Lemma~\ref{prop:homomorphism-local}, we now extend this result to a global isomorphism for all connected subgraphs with at most $ (\ell-1) $ vertices.

\begin{lemma}[Global Isomorphism]
     \label{lem:Global-homomorphism3}
    
    Let $ G_1,G_2\in\mathcal{G}_{n,m}$ be two $ k $-regular graphs with a girth of at least $ \ell $. Let $ s_1 $ and $ s_2 $ be the source vertices of $ G_1 $ and $ G_2 $, respectively. Then, there exists a rooted isomorphism $ h $ from
    $ \sigma_{G_1}(\ell-1) $ to $\sigma_{G_2}(\ell-1)$. 
\end{lemma}

\begin{proof}
We begin with the isomorphism $h: \sigma_{G_1}(\ell-1; s_1) \to \sigma_{G_2}(\ell-1; s_2)$ guaranteed by the local isomorphism lemma. Our goal is to extend this to a global isomorphism $h: \sigma_{G_1}(\ell-1) \to \sigma_{G_2}(\ell-1)$.

Let $X$ be any tree in $\sigma_{G_1}(\ell-1)$. By the local isomorphism lemma, for every vertex $v$ in $G_1$, the number of subtrees isomorphic to $X$ rooted at $v$ equals the number of such subtrees rooted at the corresponding vertex in $G_2$. Moreover, each subtree with $t$ vertices appears in exactly $t$ vertex neighborhoods (once for each choice of root vertex), so the total number of subgraphs in $G_1$ isomorphic to $X$ is equal to the total number in $G_2$.

In particular, the number of such subgraphs containing the source vertices is the same in both graphs. It follows that the number of subgraphs isomorphic to $X$ that do not contain the source vertices is likewise equal.

We can therefore extend $h$ by matching each isomorphism class of subgraphs in $G_1$ with its counterpart in $G_2$, preserving the association of subgraphs rooted at $s_1$ and $s_2$.
\end{proof}

\subsection{Collections of Trivial Cut Sets}

Having established that trivial cut sets containing at most $ (g-1) $ vertices are isomorphic in any two $ k $-regular graphs with a girth of at least $ g $, we now extend this result to collections of such cut sets. Specifically, we establish an isomorphism between collections of trivial minimal cut sets by examining families of subtrees and comparing their embeddings across two graphs, capturing different unions of trivial cut sets. In the following subsection, these embeddings will serve as a basis for evaluating the reliability of high-girth graphs.

\begin{definition}[Subtree Embedding]
Let $G$ be a $k$-regular graph with girth $g$. A \emph{subtree embedding} $\varphi\subset\sigma_G$ is a finite collection of subtrees in $G$. The \emph{skeleton} of $\varphi$ is defined as the union of all subtrees, i.e., $\bigcup_{T\in\varphi} T$.

Each tree in the collection is assigned a unique color, and the source of $ G $ is also given a distinct color. The vertices in the skeleton of $ \varphi $ are then colored by the set of all colors corresponding to trees intersecting at that vertex.

A subtree embedding $ \varphi $ induces a cut set given by
\[
\partial(\varphi) = \cup_{T \in \varphi} \partial(T),
\]

Two subtree embeddings $\varphi,\psi\subset \sigma_G$ are considered \emph{equivalent} $\varphi\sim\psi$ if they contain the same number of trees and there exists a graph isomorphism between their trees. The equivalence classes under this relation $[\psi]$ are referred to as \emph{subtree collections}. Figure~\ref{fig:embedding} illustrates subtrees and their equivalence classes.

Since we are only interested in cut sets with at most $ g(k-2) $ edges, we define the set of all subtree embeddings that induce a cut set of at most $ \omega $ edges, for $ \omega \leq g(k-2) $, as:
\begin{equation}
\mathcal{E}_{G}(\omega)  =  \{\varphi \subseteq \sigma_{G} \mid |\partial(\varphi)| \leq \omega\}.
\end{equation}
\end{definition}

We say that subtree embeddings are \emph{joint} if at least two of their subtrees, induced by the corresponding cut sets, overlap. Formally, a subtree embedding $ \varphi \in \mathcal{E}_G(\omega) $ is a \emph{disjoint subtree embedding} if 
\[
\forall T_1\neq T_2\in\varphi:\quad
T_1\cap T_2=\emptyset\quad
\text{and}\quad
\partial(T_1)\cap \partial(T_2)=\emptyset\quad
\]
If the subtree embeddings are not disjoint, they are referred to as \emph{joint subtree embeddings}. 

Disjoint embeddings of a subtree collection maximize the size of the induced cut set. Furthermore, by Proposition~\ref{prop:The_size_of_the_induced_cut_set}, all disjoint embeddings of the same subtree collection have the same number of edges in the induced cut set. Figure~\ref{fig:embedding} illustrates examples of disjoint and joint embeddings.

\begin{figure}
    \centering
    \includegraphics[width=1\linewidth]{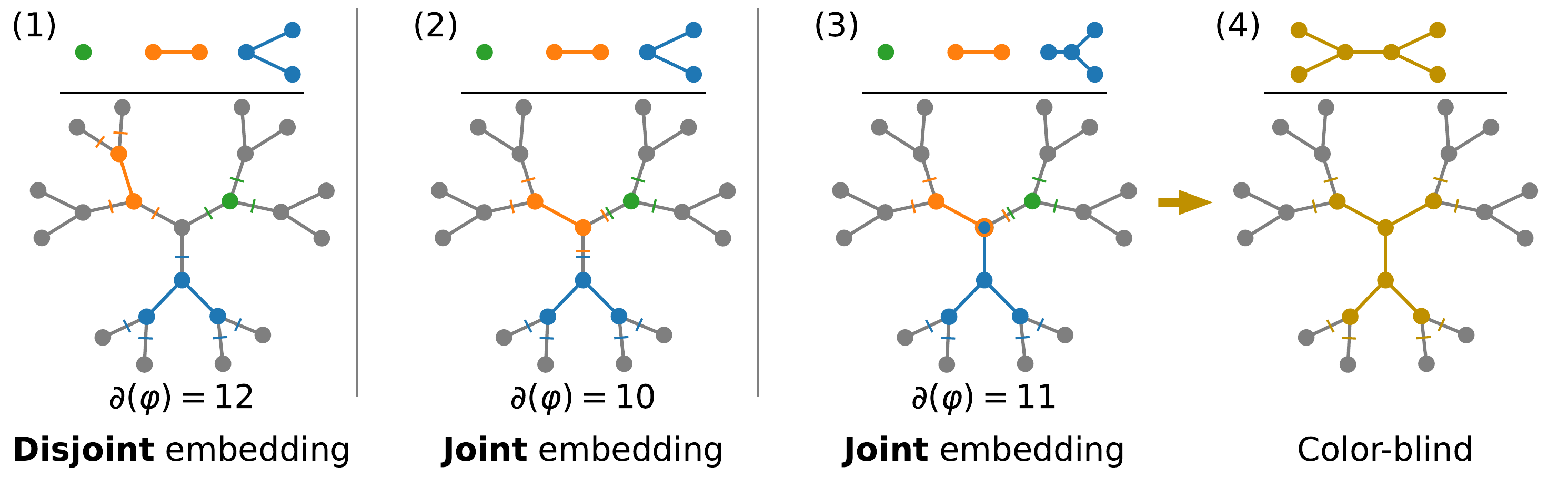}
    \caption{Illustration of subtree embeddings. (1) A disjoint subtree embedding. (2) A joint embedding with two overlapping edges. (3) A joint embedding with one overlapping edge. The upper rows represent the subtree collections, and $ \delta $ denotes the size of the induced cut set. (4) Demonstration of the color-blind principle from Lemma~\ref{lem:intersection-homomorphism} applied to the subtree embedding in (3).}

    \label{fig:embedding}
\end{figure}

The following lemma extends the global isomorphism result to establish an isomorphism between subtree embeddings in two graphs with high girth.

\begin{lemma}[Subtree Embedding Isomorphism]
\label{lem:intersection-homomorphism}
Let $ G_1 $ and $ G_2 $ be $ k $-regular graphs with the same number of vertices and edges, each having a girth of at least $ \ell $, and let $ s_1 \in V(G_1) $ and $ s_2 \in V(G_2) $ be their respective sources. Then, there exists a rooted coloring isomorphism 
\[
h: \mathcal{E}_{G_1}(\ell(k-2)) \to \mathcal{E}_{G_2}(\ell(k-2)).
\]
\end{lemma}

\begin{proof}
According to the global isomorphism, both graphs have the same set of subtree collections. Let $[\psi]$ be a subtree collection in $G_1$. Since the trees are independently embedded in the graphs, $[\psi]$ has the same number of embeddings in $G_1$, and $G_2$. Hence, we proceed with the proof by constructing a coloring isomorphism between the embedding of the given subtree collection $[\psi]$ in $G_1$ and $G_2$. The proof follows by induction on the number of subtrees in $[\psi]$.

\paragraph{Base Case.}
Suppose $|\psi|=1$, and let $\varphi=\{T\}$ be an embedding of $[\psi]$ in $G_1$. By Proposition~\ref{prop:The_size_of_the_induced_cut_set}, the subtree $T$ has fewer
than $\lfloor \ell - \tfrac{2}{k-2} \rfloor \le \ell-1$ vertices. According to the global isomorphism lemma, there exists a rooted isomorphism between the trees in $G_1$ and $G_2$ with at most $\ell-1$ vertices.

\paragraph{Inductive Step.} Assume the statement holds for all subtree collections with at most $ |\psi|-1 $ components. Consider an embedding $ \varphi $ of a subtree collection $[\psi]$ in $ G_1 $.

\subparagraph{Case 1: $ \varphi $ is a joint embedding.}  
If $ \varphi $ is joint, we recolor all intersecting subtrees in its skeleton with the same color, yielding a new disjoint embedding $ \varphi' $ with the same skeleton but a smaller induced cut set and fewer trees (see Figure~\ref{fig:embedding}). By the induction hypothesis, there exists a subtree collection $ h(\varphi') $ in $ G_2 $ that is color-isomorphic to $ \varphi' $. Then, for each disjoint tree $ X $ in $ h(\varphi') $, we can uniquely restore the original colors using the corresponding isomorphic tree in $ \varphi' $, ensuring that $ h(\varphi') $ is color-isomorphic to $ \varphi $.

\paragraph{Case 2: $\varphi$ is a disjoint embedding.}
Since the subtrees in $\varphi$ are pairwise disjoint, at most one subtree can contain the source.
Consequently, two disjoint embeddings of $[\psi]$ are color-isomorphic if and only if their
source-containing subtrees (if any) are isomorphic as rooted trees. For each subtree $T$ in $G_1$ that contains the source $s_1$, there exists a rooted isomorphic
subtree $h(T)$ in $G_2$ with respect to the source $s_2$. Because the remaining subtrees are
embedded independently of $T$, the total number of embeddings of $[\psi]$ in $G_1$ that contain
$T$ is equal to the number of embeddings in $G_2$ that contain $h(T)$. Moreover, since rooted isomorphisms also exist between joint embeddings, the number of joint embeddings containing
$T$ and $h(T)$ is the same in both graphs. It follows that the number of disjoint embeddings
of $[\psi]$ containing $T$ in $G_1$ equals the number of disjoint embeddings containing $h(T)$
in $G_2$.

We therefore construct the required color-preserving isomorphism by matching embeddings
according to the rooted correspondence(if exist) $T \mapsto h(T)$.

\end{proof}

In the following subsection, we apply the subtree embeddings isomorphism lemma to establish that graphs without non-trivial cut sets are the most reliable near zero.

\subsection{Regular Graphs Without Non-Trivial Cut Sets Are Reliable} \label{high_girth}

We decompose the unreliability polynomial into a sum of trivial and non-trivial components. Let $G$ be a $k$-regular graph with girth $g$, and let $\omega=g(k-2)$. Using the representation in Equation~\ref{eq:inc_exc_principle_ATR}, we define the \emph{trivial polynomial} as:
\begin{equation} \label{eq:ATR_trivial}
    U_G^{(t)} = \Pr \left(\bigcup_{X\in \tmcs_G}F_X\right),
\end{equation}
where $F_X$ represents the event that all the edges in $X$ fail. Since $ U_{G}^{(t)} \leq U_{G} $, we define the \emph{non-trivial polynomial} as $ U_{G}^{(n)} = U_{G} - U_{G}^{(t)} $. Equivalently, Proposition~\ref{prop:The_size_of_the_induced_cut_set} shows that each trivial cut set induced from a subtree of at most $a=\lfloor g -\frac{2}{k-2}\rfloor$, and therefore, $U_G^{(t)}$ can be expressed in terms of trees as
\begin{equation}
    U_G^{(t)} = \Pr \left(\bigcup_{T\in \sigma_G(a)}F_{\partial(T)}\right).
\end{equation}

$U_G^{(t)}$ can also be expressed in terms of subtree embeddings. For each finite collection of trivial cut sets, $\{\partial(T_i)\}_{i\in I}$, their intersection is represented by the corresponding subtree embedding $\varphi=\{T_i\}_{i\in I}$
\begin{equation}
    \bigcap_{i\in I}F_{\partial(T_i)}=F_{\partial(\varphi)}.
\end{equation}
Consequently, using the inclusion-exclusion principle, the first $\omega$ terms of the trivial unreliability polynomial are
\begin{equation} \label{eq:represent_trivial_polynom}
    U_{G}^{(t)} \overset{\omega}{=} \sum_{\varphi\in\mathcal{E}_{G}(\omega)} (-1)^{|\varphi|-1} \cdot p^{|\partial(\varphi)|},
\end{equation}
where $f\overset{\omega}{=}g$ for two polynomials indicate that $f-g=O(p^{\omega+1})$.

\begin{lemma}[Graphs with Non-Trivial Cut Sets Are Less Reliable] 
\label{lem:non-trivial_cut_Sets}  
Let $ G_1 $ and $ G_2 $ be two graphs with the same number of vertices and edges, each with a girth of at least $ g $. Suppose that $ \ntmcs_{G_1}(g(k-2)) = \emptyset $ and $ \ntmcs_{G_2}(g(k-2)) \neq \emptyset $. Then, $ G_1 $ is more reliable than $ G_2 $ near zero.
\end{lemma}

\begin{proof}
Set $ \omega = g(k-2) $. The trivial unreliability polynomial satisfies
\[
U_{G_i}^{(t)} \overset{\omega}{=} \sum_{\varphi\in\mathcal{E}_{G_i}(\omega)} (-1)^{|\varphi|-1} \cdot p^{|\partial(\varphi)|}.
\]
By applying the coloring isomorphism lemma, we conclude that 
\[
U_{G_1}^{(t)} \overset{\omega}{=} U_{G_2}^{(t)}.
\]
Since $ U_{G_1}^{(n)} = 0 $ and $ U_{G_2}^{(n)} > 0 $, it follows that $ U_{G_1} < U_{G_2} $ in the lexicographical order of the coefficient vector. By Lemma~\ref{lem:coefficients_comparison}, this implies that $ G_1 $ is more reliable than $ G_2 $ near zero.
\end{proof}

\subsection{Main Theorem} \label{subsection:main_theorem}

\begin{theorem}[Regular Graphs Without Non-Trivial Cut Sets] 
\label{thm:Most-reliable-graphs-regular}
Let $ n \geq 2 $ such that $ m=kn/2 \in \mathbb{N} $, and define $ \mathcal{D}_{n,m}(g) $ as the set of all regular graphs in $ \mathcal{G}_{n,kn/2} $ that do not contain non-trivial cut sets with at most $ g(k-2) $ edges. Suppose there exists $ g \geq 4 $ such that $ \mathcal{D}_{n,m}(g) \neq \emptyset $. Then,
\[
\mathcal{A}_{g(k-2)} = \mathcal{D}_{n, m}(g).
\]
\end{theorem}
\begin{proof}
    Set $\omega=g(k-2)$, and suppose that $G\in \mathcal{D}_{n,m}(g)$. Suppose also that $G'\in \mathcal{G}_{n,m}$ is another graph with the same size. If $g'<g$, then according to Lemma~\ref{lem:non-trivial_cut_Sets}, $U_{G}\overset{\omega}{<}U_{G'}$. However, if $g=g'$, their non-trivial polynomial maintain $U_{G}^{(n)}\overset{\omega}{=}U_{G'}\overset{\omega}{=}0$, and by Equation~\eqref{eq:represent_trivial_polynom}, their trivial polynomials maintain $U_G^{(t)}\overset{\omega}{=}U_{G'}^{(t)}$, leading to $U_G\overset{\omega}{=}U_{G'}$.
\end{proof}

\begin{corollary} \label{cor:no-ntmcs-optimal}
If there exists a $k$-regular graph $G\in\mathcal{G}_{n,m}$ such that $\ntmcs_G(g(k-2)) = \emptyset$, then the most reliable $k$-regular graph near zero in $\mathcal{G}_{n,m}$ does not contain non-trivial cut sets of size at most $g(k-2)$.
\end{corollary}

An immediate consequence of the Corollary is that by identifying at least one regular graph without non-trivial cut sets, we can establish a lower bound on the girth of the most reliable graph near zero, and by that, lower the search space.

\begin{lemma} \label{lem:girth6}
The girth of the most reliable cubic graphs near zero with at least 14 vertices is at least 6.
\end{lemma}

\begin{proof}
We construct a cubic graph with girth 6 that contains no non-trivial cut sets with at most 6 edges. This is achieved by first drawing a cycle with $ n $ vertices. Then, each even-indexed vertex $ i $ is connected to $ (i+5) \mod n $. Examples of such graphs are shown in Figure~\ref{fig:most_reliable_graphs}, where $ r $ is 8 or 9.
\end{proof}

We encourage further exploration of families of regular graphs with high girth and without non-trivial cut sets to improve the lower bound on the girth of the most reliable graph near zero.

Theorem~\ref{thm:Most-reliable-graphs-regular} is used in the next Section to find infinite families of most reliable sparse graphs near zero.

\section{Reliable Sparse Graphs}    
\label{sec:sparse}
\subsection{Reliability Near Zero}
\label{subsec:unequal_chains}

In this section, we study sparse graphs with constant redundancy $r \leq 0.5n$, an asymptotically large number of vertices $n$, and a small failure probability $p$. As shown in Subsection~\ref{subsec:sparse}, sparse graphs that minimize the first two reliability coefficients $b_1$ and $b_2$ have a 3-connected cubic structure graph, with chains whose lengths differ by at most one. Constructing a highly reliable sparse graph in this regime involves two main steps: (1) identifying a reliable cubic structure graph, and (2) determining the placement of long chains.

We recall the decomposition:
\[
U_G = U_G^{(1)}+U_G^{(2)}
\]
where $U_G^{(2)}$ is the probability of disconnection due to two or more failures in the same chain, and depends only on the number and lengths of the chains. The term $U_G^{(1)}$ accounts for the probability of a failure in which each chain contains at most one failing edge, and thus captures the structural contribution of the graph.

We express $U_G^{(1)}$ as:
\begin{equation}
    U_G^{(1)}=\sum_{k=0}^{3(r-1)} b_k^{(1)}p^k(1-p)^{m-k},
\end{equation}
and by expressing $b_k^{(1)}$ as a polynomial in $c$,
\begin{equation}
    b_k^{(1)}=\sum_{j=0}^k\, \gamma_{k,j}\,c^{k-j},
\end{equation}
we get

\begin{equation}
U_G^{(1)} = \sum_{k = 0}^{3(r - 1)} \sum_{j = 0}^{k} \gamma_{k,j} \cdot c^{k - j} p^k (1 - p)^{m - k},
\end{equation}
where $\gamma_{k,j}$ is calculated by summing for each subset of $j$ long chains, the number of cut sets with $k$ chains that contain the long chains subset. These coefficients depend only on the structure graph and the placement of the long chains, not on their lengths.

Lemma~\ref{lem:Extension of coefficients comparison} extends the classical coefficient compression argument. It shows that for any $r \geq 2$, there exists $n_0$ such that for all $n > n_0$, the asymptotically most reliable sparse graph near zero minimizes the coefficients $\gamma_{k,j}$ in lexicographic order on the index pairs $(k,j)$.

This leads to a key structural reduction: in the asymptotical case where $n$ is large enough, the most reliable sparse graph with $r$ redundant edges and $n = 2(r - 1) + (c - 1) \cdot 3(r - 1) + \lambda$ vertices depends only on its marked structure, that is, a cubic structure graph together with a choice of $\lambda$ long chains.

\medskip

We next consider a minimization of $U_G$ in a small interval that is not near zero. Lemma~\ref{lem:transpose_coeff_comp} shows that for large $n$, the uniformly most reliable graph must minimize the coefficients $\gamma_{k,j}$ in transposed lexicographic order on $(j,k)$ to be the most reliable on some interval of $p$. The key here, is that while in the regular lexicographical order, we set a high $c$, and for that $c$, we sufficiently lower $p$, here we set $s=c\frac{p}{1-p}$ to be the new $p$ variable, in that case, $cp$ is approximately a constant and we increase $c$, while ensuring $p\sim\frac{s}{c}$, leading to the minimization in the reverse lexicographical order.

\begin{lemma}[Transposed Coefficient Compression]
\label{lem:transpose_coeff_comp}
For each $r \geq 2$, there exists $n_0$ such that for all $n > n_0$, the uniformly most reliable sparse graph with $n$ vertices and $r$ redundant edges has a marked structure that minimizes the coefficients $\gamma_{k,j}$ in transposed lexicographic order on $(j,k)$.
    \[
\begin{tikzpicture}[every node/.style={anchor=center}]
  \matrix (m) [matrix of math nodes, row sep=1em, column sep=2em] {
    \gamma_{0,1} & \gamma_{1,1} &           &           \\
    \gamma_{0,2} & \gamma_{1,2} & \gamma_{2,2} &         \\
    \gamma_{0,3} & \gamma_{1,3} & \gamma_{2,3} & \gamma_{3,3} \\
    \vdots       & \vdots       & \vdots       & \vdots    \\
  };

  \draw[-{Stealth}] (m-1-1) -- (m-2-1);
  \draw[-{Stealth}] (m-2-1) -- (m-3-1);
  \draw[-{Stealth}] (m-3-1) -- (m-4-1);
  \draw[-{Stealth}] (m-4-1) -- (m-1-2);
  
  \draw[-{Stealth}] (m-1-2) -- (m-2-2);
  \draw[-{Stealth}] (m-2-2) -- (m-3-2);
  \draw[-{Stealth}] (m-3-2) -- (m-4-2);
  \draw[-{Stealth}] (m-4-2) -- (m-2-3);

  \draw[-{Stealth}] (m-2-3) -- (m-3-3);
  \draw[-{Stealth}] (m-3-3) -- (m-4-3);
  \draw[-{Stealth}] (m-4-3) -- (m-3-4);

  \draw[-{Stealth}] (m-3-4) -- (m-4-4);

\end{tikzpicture}
\]

\end{lemma}

\begin{proof}
We make a change of variables. Let $s = \frac{p}{1 - p} c$ and let $c = s^{-t}$. Then the polynomial becomes:
\begin{equation}
U_G^{(1)} = \left(\frac{s^{-t}}{s^{-t}+s}\right)^m\cdot \sum_{k = 0}^{3(r - 1)} \sum_{j = 0}^{k} \gamma_{k,j} s^{k + jt}
\end{equation}
Moreover, each failure of $k>r$ chains results in a disconnection in any graph. Therefore, the $\gamma_{k,j}$ coefficients for $k>r$ are only dependent on $r$, leading to
\begin{equation}
U_G^{(1)} = \left(\frac{s^{-t}}{s^{-t}+s}\right)^m\cdot \sum_{k = 0}^{r} \sum_{j = 0}^{k} \gamma_{k,j} s^{k + jt}+\rm{Const}(r)
\end{equation}

Since this is a polynomial in $s$, and $s$ can be made arbitrarily small by taking $p$ sufficiently close to zero, minimizing $U_G^{(1)}$ corresponds to minimizing the coefficients of the lowest $s$-degree terms first. That is, we prefer $\gamma_{k,j}$ over $\gamma_{k',j'}$ whenever $k + jt < k' + j't$. For $k\leq r$ and by choosing $t$ such that $t>r$, we get 
\begin{equation}
\label{eq:kj}
k+jt < k'+j't \quad\Longleftrightarrow\quad
\big(j<j'\big)\;\; \text{or}\;\; \big(j=j' \text{ and } k<k'\big),
\end{equation}
which is exactly the reverse lexicographical order.

To find the interval in which the Lemma holds, assume that $c$ is constant and compare $U_G$ with $U_H$, where $G$ is a marked structure that is optimal in the transposed lexicographical order, and $H$ is another marked graph of the same size. From Equation~\eqref{eq:kj}, there exist small $s_0$ such that if $s<s_0$, and $t>t_0=r$, then $U_G(s,t)<U_H(s,t)$.

We left to check if for a constant $c$, there exists an interval of $p$ such that $t>t_0$, and $s<s_0$. from the definition of $s$,
\[
s=\frac{p}{1-p}c< s_0\Rightarrow p< \frac{s_0}{s_0+c}
\]
and from the $t$ equation
\[
t=-\frac{\log(c)}{\log(\frac{p}{1-p}c)}> t_0\Rightarrow p>\frac{1}{1+c^{1+\frac{1}{t_0}}}
\]

Together, we get that the lemma holds in the interval
\begin{equation}
\frac{1}{1+c^{1+\frac{1}{t_0}}}< p< \frac{s_0}{s_0+c}
\end{equation}

This inequality holds only if 
\[
    \frac{1}{1+c^{1+\frac{1}{t_0}}} <\frac{s_0}{s_0+c} 
\]

which means
\begin{equation}
c> s_0^{-t_0}
\end{equation}
Those, setting $c_0=s_0^{-t_0}$, the Lemma holds if $c>c_0$.
\end{proof}

Because $\gamma_{k,0} = b_k$, we conclude that the structure of the graph that minimizes the sequence $(\gamma_{k,j})$ in transposed lexicographic order must coincide with the most reliable cubic graph near zero.

\medskip

To minimize $\gamma_{k,1}$ and $\gamma_{k,2}$, we analyze how cut sets interact with long chains. For $\gamma_{k,1}$, we are interested in the number of $k$-edge cut sets that contain a specific long chain. For $\gamma_{k,2}$, we must count the number of pairs of long chains that appear together in such a cut set.

Fortunately, in cubic structure graphs with only trivially small cut sets, the number of such cut sets is a function of the distance between the chains. This is formalized in the following lemma.

\begin{lemma}[Cut Sets Containing Two Edges]
\label{lem:cut_set_contain_two_edges}
Let $G$ be a cubic graph in which all cut sets of size at most $g$ are trivial. Let $d \geq 0$ and $k$ satisfy $d - 2 \leq k \leq g$. Then, for any two edges at distance $d$, the number of $k$-edge cut sets containing both edges is the same.
\end{lemma}

\begin{proof}

The result follows from the coloring isomorphism in Lemma~\ref{lem:intersection-homomorphism}. Given two edges at distance $d$, choose one of them and let its closer endpoint serve as the source vertex for the isomorphism. Let $G$ and $H$ be two cubic graphs with the same conditions of the lemma, each of which contains a pair of long chains of distance $d$. By the coloring isomorphism lemma, any collection of connected subgraphs rooted at the source of $G$ and inducing a $k$-edge cut set has an isomorphic collection in $H$. Therefore, the number of $k$-edge cut sets containing both edges is the same on both graphs.
\end{proof}

This lemma implies two corollaries when $k<g$:

1. The value of $\gamma_{k,1}$ is invariant under the placement of long chains.

2. Minimizing $\gamma_{k,2}$ is equivalent to minimizing the number of long chain pairs at distance $k - 2$.

With these structural observations established, we are now equipped to examine the two minimization frameworks. We first analyze the minimization of coefficients in the $(k,j)$ lexicographic order, and then move on to the transposed $(j,k)$ order.

\subsubsection{$(k,j)$ Minimization}

\begin{definition}[Separation vector]
Let $G$ be a marked structure graph whose smallest non-trivial cut set has size $g$. The \emph{separation vector} is defined as
\[
\mu = (\mu_1, \mu_2, \dots, \mu_{g-3}),
\]
where $\mu_d$ denotes the number of long chain pairs at distance $d$ in $G$.
\end{definition}

We now show that minimizing $\gamma_{k,2}$ for small values of $k$ is equivalent to minimizing the number of long chain pairs at distance $k-2$.

\begin{lemma}[Minimizing $\gamma_{k,2}$]
\label{lem:minimize_gamma_k2}
Let $G$ be a cubic graph and let $g$ be the size of the smallest non-trivial cut set. Denote by $[G]_{k,j}$ the set of marked graphs in $\mathcal{A}_{k,j}$ with structure graph $G$. Then
\[
H \in [G]_{k,2} \iff H \in [G]_{k,0} \text{ and } \mu_{k-2} \text{ is minimized over all graphs in } [G]_{k,0}.
\]
\end{lemma}

\begin{proof}
We saw that the coefficient $\gamma_{k,1}$ for $k<g$ is independent of the placement of long chains. Therefore, $\mathcal{A}_{k,0} = \mathcal{A}_{k,1}$.

Next, decompose $\gamma_{k,2} = \sum_{i=1}^{k-2} \gamma_{k,2}^{(i)}$, where each $\gamma_{k,2}^{(i)}$ sums over long chain pairs at distance $i$ the number of $k$-cut sets containing them. Lemma~\ref{lem:cut_set_contain_two_edges} implies that $\gamma_{k,2}^{(i)}$ depends only on $\mu_i$ if $k<g$.

For the base case $k = 3$, we have $\gamma_{3,2} = \gamma_{3,2}^{(1)}$, and minimizing $\mu_1$ suffices. Assuming the lemma holds for $k = a-1$, we consider $k = a$. From the induction assumption, all graphs in $[G]_{a,0}$ have identical values of $\mu_i$ for $i < a-2$ and thus equal $\gamma_{a,2}^{(i)}$ for those terms. The only remaining term to minimize is $\gamma_{a,2}^{(a-2)}$, which depends solely on $\mu_{a-2}$.
\end{proof}

To minimize $\gamma_{k,j}$ for $j \geq 3$, more complex configurations involving triplets, quartets, etc., must be considered. We restrict our discussion to the case $k \leq 4$.

Minimizing the first three coefficients $b_1^{(1)}, b_2^{(1)}, b_3^{(1)}$ is well established (see \cite{wang1997structure,landgren2024newframeworkidentifyingreliable}), but we reframe it here using our terminology. As shown earlier, a cubic structure graph with chain lengths differing by at most one yields $b_1^{(1)} = b_2^{(1)} = 0$. To minimize $\gamma_{3,0} = b_3^{(1)}$, the structure must contain only trivial 3-edge cut sets (i.e., incident edges of a node).

In this setting, each long chain belongs to exactly two cut sets, so $\gamma_{3,1} = 2\lambda$, independent of their placement. A pair of long chains is in the same cut set only if adjacent to the same node. Let $\alpha_i$ be the number of structural vertices adjacent to $i\in\{0,1,2,3\}$ long chains. Then
\begin{equation}
\label{eq:b3}
b_3^{(1)} = nc^3 + 2\lambda c^2 + (\alpha_2 + 3\alpha_3)c + \alpha_3.
\end{equation}

\begin{lemma}[\cite{wang1997structure,landgren2024newframeworkidentifyingreliable}]
\label{lem:b3_sparse}
Let $G$ be a sparse graph with $r$ redundant edges and $\lambda$ long chains. $G\in \mathcal{A}_3$ if and only if it has a marked structure graph which is 3-connected and cubic with only trivial 3-edge cut sets, and the number of long chains adjacent to each vertex differs by at most 1. More precisely:
\begin{enumerate}
    \item If $\lambda \leq r - 1$, the long chains form a matching (no two long chains are adjacent to the same vertex).
    \item If $r - 1 < \lambda \leq 2(r - 1)$, then exactly $2(\lambda - (r - 1))$ vertices are adjacent to two long chains, and the remaining vertices are adjacent to one.
    \item If $\lambda > 2(r - 1)$, the short chains form a matching.
\end{enumerate}
\end{lemma}

\begin{proof}
It follows from equation~\eqref{eq:b3} that a graph satisfying the lemma’s conditions lies in $\mathcal{A}_3$. By Petersen’s theorem~\cite{petersen1891theorie}, any bridgeless cubic graph has a perfect matching.

Since the structure graph has $2(r-1)$ vertices, this allows us to place $r-1$ long chains on the matching when $\lambda \leq r-1$, and symmetrically, place the short chains on the matching when $\lambda > 2(r-1)$. For intermediate values, $r-1 < \lambda \leq 2(r-1)$, we can begin with one perfect matching for $r-1$ long chains. After removing those edges, the remaining graph is 2-regular. If it decomposes into even-length cycles, we can find a second matching to host the extra $\lambda - (r-1)$ long chains, so that no vertex has more than two incident long chains.

Although not all cubic graphs guarantee this condition, we can always find at least one graph that satisfies the required properties. For example, take a cycle with $2(r-1)$ vertices and place the chords $i \leftrightarrow i + r-1 \mod 2(r-1)$ as the first perfect matching. This ensures both the required girth and a structure that supports two-edge matching.
\end{proof}

Note that the above Lemma works not only in the asymptotic case, as minimizing $b_3^{(1)}$ is solely dependent on minimizing $\alpha_2$.

Now consider minimizing $b_4^{(1)}$. To minimize $\gamma_{4,0}$, the structure graph must have a girth of at least 5(if possible) and contain only trivial cut sets. Lemma~\ref{lem:girth6} guarantees the existence of such graphs for $r \geq 8$, and for $r = 6, 7$ this is confirmed by the most reliable graphs presented in Figure~\ref{fig:most_reliable_graphs}. 

Minimizing $\gamma_{4,2}$ is equivalent to minimizing $\mu_2$, the number of long chain pairs at distance 2. To minimize $\gamma_{4,3}$, we must reduce the number of long chain triplets that participate in a 4-edge cut set. Each trivial minimal cut-set with four edges is formed by the four chains adjacent to a given chain. If three of them are long, they contribute 1 to $\gamma_{4,3}$; if all four are long, they contribute 4.

Let $\beta_i$ denote the number of chains that are adjacent to exactly $i\in\{0,1,2,3,4\}$ long chains. Then, to minimize $\gamma_{4,3}$ one must minimize $\beta_3 + 4\beta_4$, and to minimize $\gamma_{4,4}$ one must minimize $\beta_4$.

\medskip

Based on the analysis above, we can summarize the structural conditions of the asymptotically most reliable marked structure near zero.

\begin{lemma}[Properties of the asymptotically most reliable marked structure near zero]
\label{lem:properties_sparse}
Suppose $r \geq 2$. There exists $n_0$ such that for all $n > n_0$, any most reliable graph near zero with $n$ vertices and $r$ redundant edges has a marked structure graph that satisfies:
\begin{enumerate}
    \item The structure is cubic and 3-connected, with only trivial 3-edge cut sets and chain lengths differing by at most 1.
    \item The girth is at least $4$ if $r \geq 4$, and at least $5$ if $r \geq 6$.
    \item The following parameters are minimized sequentially:
    \[
    \xrightarrow{r \geq 4} \mu_1 \xrightarrow{r \geq 5} \mu_2 \rightarrow \beta_3 + 4\beta_4 \rightarrow \beta_4 \xrightarrow{g \geq 6} \mu_3.
    \]
    The arrows indicate that each parameter should be minimized only if the condition above the arrow holds. For example, $\mu_1$ is minimized when $r \geq 4$, and $\mu_3$ is minimized only if a graph satisfying all previous conditions exists and has girth at least 6.

\end{enumerate}
\end{lemma}
The parameters minimized in Lemma~\ref{lem:properties_sparse} refer to the \emph{reliability obstruction parameters}. An example of applying this Lemma to find the optimal long-chain configuration is presented in Subsection~\ref{subsec:marked}. 

\subsubsection{$(j,k)$ Minimization}
In contrast to the $(k,j)$ order, minimizing the coefficients $\gamma_{k,j}$ lexicographically in the transposed order $(j,k)$ simplifies the analysis, as it allows us to focus primarily on the case $j \leq 2$.

As established in previous sections, the structure that minimizes all $\gamma_{k,0} = b_k$ terms sequentially is the most reliable cubic graph with $2(r - 1)$ vertices near zero. Furthermore, when the structure graph has only trivial cut sets of size at most $k$, the coefficients $\gamma_{k,1}$ are invariant under the arrangement of long chains. Therefore, the first non-trivial dependence on the chain placement arises in $\gamma_{g,1}$, where $g$ is the size of the smallest non-trivial cut set in the structure graph.

To minimize the coefficients $\gamma_{k,1}$ sequentially in $k$, we assign a rank to each chain and select the $\lambda$ chains with the lowest ranks. For each chain $e$, define its reliability vector as
\[
\rho^{(e)} := (\rho_3^{(e)}, \rho_4^{(e)}, \dots, \rho_{3(r-1)}^{(e)}),
\]
where $\rho_k^{(e)}$ denotes the number of $k$-edge cut sets in the structure graph that include the chain $e$. Then, the total contribution to $\gamma_{k,1}$ in the marked structure graph $G$ from all chains is given by
\[
\gamma_{k,1} = \sum_{e \in E(G)} \rho_k^{(e)}.
\]
Let $\mathrm{rank}(e)$ denote the lexicographic rank of $\rho^{(e)}$ among all chains in the structure graph. The $\lambda$ long chains should then be assigned to the $\lambda$ chains with the smallest ranks. For example, if $\lambda = 7$, and there are 4 chains of rank 1 and 5 chains of rank 2, then the 4 rank-1 chains must be selected, along with 3 of the 5 rank-2 chains.

As stated above, minimizing $\gamma_{k,2}$ for $k < g$ is equivalent to minimizing $\mu_{k-2}$, the $(k-2)$-entry of the separation vector. We conclude that the optimal arrangement of long chains under $(j,k)$ minimization must:
\begin{enumerate}
    \item Assign the long chains to the $\lambda$ chains with lowest $\mathrm{rank}(e)$;
    \item Among all such assignments, minimize the separation vector $\mu$ lexicographically.
\end{enumerate}

We summarize these observations in the following lemma:

\begin{lemma}[$(j,k)$ Minimization Criteria]
\label{lem:minimizing_reverse}
Let $r \geq 2$ and $0 \leq \lambda < 3(r - 1)$. Then any marked structure graph that minimizes the coefficients $\gamma_{k,j}$ in transposed lexicographic order $(j,k)$ must satisfy:
\begin{enumerate}
    \item The structure graph is the most reliable cubic graph with $2(r - 1)$ vertices near zero.
    \item The long chains are assigned to the $\lambda$ edges with the lowest reliability load ranks $\mathrm{rank}(e)$.
    \item Among all such long chains assignments, the separation vector $\mu$ is minimized in lexicographic order.
\end{enumerate}
\end{lemma}

In the following Subsection~\ref{subsec:maximize_the_tree_number}, we show that the structure graphs of uniformly most reliable graphs in the asymptotic regime typically exhibit symmetry. In such cases, the $\mathrm{rank}(e)$ values are often equal across all chains, reducing the optimization to the minimization of the separation vector alone.

\subsection{Maximizing the Tree Number}
\label{subsec:maximize_the_tree_number}

We now shift our attention to finding sparse graphs that maximize the tree number, i.e., are most reliable near one. This result can establish the non-existence of a uniformly most reliable graph. We do so by constructing a graph in which chain lengths differ by more than one, which achieves a higher number of spanning trees. Let the number of edges in each chain $e_i$ be $c+r_i$, where $r_i$ can take any integer value, including negative values.

Similar to the extended coefficient comparison method, the tree number can be expressed as a polynomial in the variable $c$. Each spanning tree of the graph is induced from a spanning tree of its structure graph. Moreover, each spanning tree can contain at most one failing edge from a given chain. Consequently, each spanning tree $T$ of the structure graph contributes $ \prod_{e_i \notin E(T)} (c+r_i) $ spanning trees to the total count in the whole graph. 

Since each such product contains exactly $ r = m - (n-1) $ terms, summing over all spanning trees of the structure graph and expanding the multiplication terms yields the following formula for the total number of spanning trees:

\begin{equation}
T(G) = \sum_{X \subseteq E(S(G))} \left( T(S(G) - X) \prod_{e_i \in X} r_i \right) \cdot c^{r - |X|}
\end{equation}
where $ T(S(G) - X) $ represents the number of spanning trees in the structure graph that do not contain the chains in $ X $.

If $ \{r_i\} $ are constants and $ c $ is large, the coefficients of the higher powers of $c$ dominate. The coefficient of $ c^r $ is $ T(S(G)) $, representing the number of spanning trees in the structure graph. Consequently, if the most reliable cubic graph near zero does not maximize the number of spanning trees, then for sufficiently large $ c $, a uniformly most reliable graph cannot exist.

\begin{lemma}[The Structure of $ t $-Optimal Graph] 
    Let $r >1$, and assume that the most reliable cubic graph with $ 2(r-1) $ vertices near zero does not maximize the number of spanning trees among graphs with the same number of vertices and edges. Then, there exists $ n_0 $ such that for all $ n > n_0 $, no uniformly most reliable graph exists with $ n $ vertices and redundancy $ r $.
\end{lemma}

The second most dominant factor in the number of spanning trees is given by:
\[
(c+1)^{r-1} \cdot \sum_{e_i\in E(S(G))} T(S(G) - e_i) r_i.
\]
Since the sum of $ \{r_i\} $ is constant, maximizing this quantity requires increasing $ r_i $ for chains with a high number of spanning trees $ T(S(G) - e_i) $. 

Consider two chains $ e_1 $ and $ e_2 $ such that $ T(S(G) - e_1) > T(S(G) - e_2) $. In this case, we say that the structure graph $ S(G) $ is not \emph{tree-balanced}. If a graph is the most reliable near zero, then $ r_i \in \{0,1\} $. However, if the structure graph is not tree-balanced, we can increase $ r_1 $ and decrease $ r_2 $, thereby increasing the coefficient of $ (c+1)^{r-1} $. This adjustment yields a graph with more spanning trees when $ c $ is large. Consequently, if the most reliable structure graph near zero is not tree-balanced, a uniformly most reliable graph does not exist for sufficiently large $ c $.

\begin{lemma}[Non-Tree-Balanced Graphs] \label{lem:non_UMRG_not_t_balanced}
    Let $ r > 1 $, and assume that the most reliable cubic graph near zero with $ 2(r-1) $ vertices is not tree-balanced. Then, there exists $ n_0 $ such that for all $ n > n_0 $, no uniformly most reliable graph exists with $ n $ vertices and $ r $ redundant edges.
\end{lemma}

In Figure~\ref{fig:most_reliable_graphs}, we analyze each reliable graph near zero to determine whether it is tree-balanced. This analysis reveals infinite families of graphs with the same redundancy and a high number of vertices for which a uniformly most reliable graph does not exist. By applying Kirchhoff's theorem, we efficiently compute the tree number, confirming that graphs with a low number of redundant edges tend to be balanced, whereas those with a high number of redundant edges are typically not. This highlights the crucial role of symmetry in the existence of uniformly reliable graphs.

\section{Identifying Uniformly Most Reliable Graphs}
\label{sec:most_reliable}

\subsection{Framework Overview}
Our goal is to determine, for a given redundancy level $r$ and number of long chains $0 \leq \lambda < 3(r - 1)$, whether a uniformly most reliable graph exists in the asymptotic regime (i.e., for sufficiently high number of nodes $n$), and if so, to identify its marked structure.

The method proceeds in three stages:
\begin{enumerate}
    \item Identify candidate cubic graphs on $2(r-1)$ vertices that are most reliable near zero.
    \item For each candidate, determine placements of $\lambda$ long chains that independently minimize the reliability rank $\rho$, the separation vector $\mu$ and the reliability obstruction parameters from Lemma~\ref{lem:properties_sparse}. A placement achieving both minima is called an \emph{optimal long-chain configuration}.
    \item Verify whether the resulting structure graph is tree-balanced and admits an optimal long-chain configuration. If both conditions hold, the graph is a candidate for being uniformly most reliable. Otherwise, this contradicts the existence of a uniformly most reliable graph for the given values of $r$ and $\lambda$ and a sufficiently large number of nodes $n$. 
\end{enumerate}

\subsection{Algorithmic Procedure}
\label{subsec:algorithm}
The following algorithm formalizes the identification process for the most reliable marked structure graph.

\begin{algorithm}[H]
\caption{Identification of a uniformly most reliable marked structure}
\label{alg:most_reliable}
\begin{algorithmic}[1]
\Require Redundancy level $r$ and number of long chains $\lambda$
\State Select a cubic graph $G$ on $2(r-1)$ vertices that is most reliable near zero,
obtained via enumeration of maximal-girth cubic graphs.
\If{$G$ is not tree-balanced or does not maximize the number of spanning trees}
    \State \Return No uniformly most reliable marked structure exists.
\EndIf
\State Let $E_{\mathrm{opt}} \subseteq E(G)$ be the set of edges attaining the minimal values of $\rho^{(e)}$.
\State Compute the minimal separation vector $\mu$ among all $\lambda$-edge subsets of $E(G)$.
\State Compute the minimal reliability obstruction parameters (Lemma~\ref{lem:properties_sparse}) among all $\lambda$-edge subsets of $E(G)$.
\State Attempt to find a $\lambda$-edge subset $S \subseteq E_{\mathrm{opt}}$ that simultaneously attains both minima.
\If{no such subset exists}
    \State \Return No uniformly most reliable marked structure exists.
\EndIf
\State Let $\gamma_{k,j}$ denote the coefficient vector induced by the marked structure $(G,S)$.
\If{there exists a cubic graph $G'$ on $2(r-1)$ vertices whose optimal marked structure yields a strictly smaller coefficient vector than $\gamma_{k,j}$ in lexicographic order}
    \State \Return No uniformly most reliable marked structure exists.
\Else
    \State \Return $(G,S)$ is a candidate uniformly most reliable marked structure.
\EndIf
\end{algorithmic}
\end{algorithm}

\subsection{Computational Complexity and Graph Enumeration}
The computational complexity of the method depends primarily on the redundancy parameter $r$. The structure graph used in the sparse case has $n_s = 2(r - 1)$ vertices. The algorithm is linear in the number of vertices of the original graph, and exponential in the number of redundant edges.

The main computational bottleneck is generating all $3$-regular graphs with prescribed girth $g$. For fixed $g$, the number of such graphs with $n_s$ vertices and girth at least $g$ grows as $\Theta(e^{\alpha_g n_s})$ for some constant $\alpha_g > 0$~\cite{mckay1990asymptotic}. Explicit enumeration is tractable for small $n_s$ using isomorphism-free generation tools such as \texttt{genreg}~\cite{meringer1999fast} and \texttt{snarkhunter}~\cite{Brinkmann2015GenerationOC}, whose running time is polynomial per generated graph and therefore dominated by the total number of outputs.

Since our algorithm only considers graphs with maximal girth, the number of graphs to enumerate may, in fact, decrease as the redundancy parameter increases. For example, there are $385$ cubic graphs with $22$ vertices and girth $6$ (which is maximal), but only a single graph with $24$ vertices and girth $7$, namely the unique $(3,7)$-cage. Some of the results obtained via these enumeration algorithms are available through the House of Graphs database~\cite{coolsaet2023house}.

Once all cubic graphs with maximal girth are enumerated, the algorithm verifies which of them contain non-trivial cut sets. This involves checking all connected subgraphs of size at most $n_s/2$, which is exponential in $n_s$.

The unreliability polynomial of the structure graph can be computed exactly in $O(2^{3r})$ time. If only the first $t < 3(r - 1)$ terms are required, this reduces to $O(2^t)$. Moreover, for graphs with small treewidth, tree-decomposition-based algorithms can be applied to compute the unreliability polynomial more efficiently, as discussed in the next subsection.

To identify long-chain configurations that minimize both the separation vector and the parameters defined in Lemma~\ref{lem:properties_sparse}, we formulate the problem as a linear programming optimization. The construction of the linear program and details of its implementation are described in Section~\ref{subsec:marked}.

In summary, the algorithm's complexity is governed primarily by $r$ and $\lambda$ and depends only weakly on $n$. In fact, the total number of graphs to consider may decrease as $r$ increases, due to the scarcity of graphs with maximal girth. For small values of $r$, the algorithm performs well in practice. Determining the number of cubic graphs with maximal girth for arbitrary vertex counts remains an interesting open problem.

\subsection{Database of Cubic Graphs}
\label{subsec:cubic_db}
Another key contribution of our research is the establishment of a reliability database. This database contains 5,535 cubic graphs, along with their properties and unreliability polynomials. It serves as a valuable tool for testing hypotheses and applying the algorithms developed in this study, such as computing the most reliable graphs near zero and assessing the non-existence of a uniformly most reliable graph.

The database is publicly available, enabling other researchers to explore unreliability polynomials and investigate additional research questions, such as identifying uniformly optimal graphs and detecting the roots and inflection points of the reliability curves, which is itself an important topic of investigation~\cite{Agrawal1984ASO,RootsRel}. The initial collection of cubic graphs is taken from the House of Graphs database~\cite{coolsaet2023house}.

Constructing such a database is challenging because computing the reliability of a single graph is exponential in complexity. To address this, we utilized an algorithm based on tree decomposition~\cite{goharshady2020efficient}, which computes graph reliability with a complexity of $O(n \cdot \omega^3 \cdot 2^{(\omega+2)(\omega+1)} )$, where $ \omega $ is the treewidth of the graph. In sparse graphs, the treewidth is typically small, leading to a significant improvement in computational performance.

Our primary use of the database is the finding of the most reliable cubic graphs near zero, summarized in Figure~\ref{fig:most_reliable_graphs}, and the identification of the marked structure graphs that are the only candidates to be the uniformly most reliable graphs in the asymptotic regime, or to prove the non-existence of such graphs. In the next subsection, we show how to use the database to identify the only candidates for the uniformly most reliable graphs.

\subsection{Results on Most Reliable Marked Structures}
\label{subsec:marked}
We apply Algorithm~\ref{alg:most_reliable} to identify asymptotically most reliable graphs. As shown in Figure~\ref{fig:most_reliable_graphs}, the only tree-balanced cubic graphs that are most reliable near zero occur for $r \in \{2,3,4,6,8,9,16\}$. The cases $r \leq 4$ have already been resolved in prior work~\cite{boesch1991existence,wang1994proof,romero2017building,Rela2017PetersenGI,canale2019building}. This leaves the sparse asymptotic cases $r \in \{6,8,9,16\}$ for further investigation.

For $r \leq 9$, our database includes all cubic graphs with fewer than 9 redundant edges. This allows us to exhaustively examine the most reliable structure graphs and determine whether the optimal cubic graph near zero also maximizes the tree number, and whether its long chain configuration minimizes the terms in Lemma~\ref{lem:properties_sparse}. Our initial hypothesis was that graphs with lower girth might be more reliable in the asymptotic regime, as the presence of shorter cycles could improve the separation vector among long chains, potentially leading to higher reliability near $p=0$.

Consistent with this hypothesis, we found an example in the case $r = 8$ in which a marked structure graph with lower girth outperforms the most reliable structure near zero. Specifically, for $\lambda \in \{4,17\}$, Figure~\ref{fig:long_chains} shows two candidates: graph (1), with girth 5 and separation coefficient $\mu_2 = 0$, and graph (2), with girth 6 but separation coefficient $\mu_2 = 1$. Although the second graph has a more reliable structure graph, the first graph is asymptotically more reliable near zero due to superior separation.

This advantage arises because the 2-ball around the long chains $(0,4)$ and $(4,12)$ in graph (1) includes only 17 chains, while the 2-ball around $(4,9)$ and $(6,7)$ in graph (2) includes 19. The key topological feature responsible for this is the presence of two interconnected 5-cycles in graph (1), highlighted in black. This fused cycle configuration yields a tighter 2-neighborhood and better separation among chains. We hypothesize that such intertwined 5-cycle motifs are instrumental in designing 5-girth structure graphs with optimal separation vectors.

\begin{figure}
    \centering
    \includegraphics[width=0.7\linewidth]{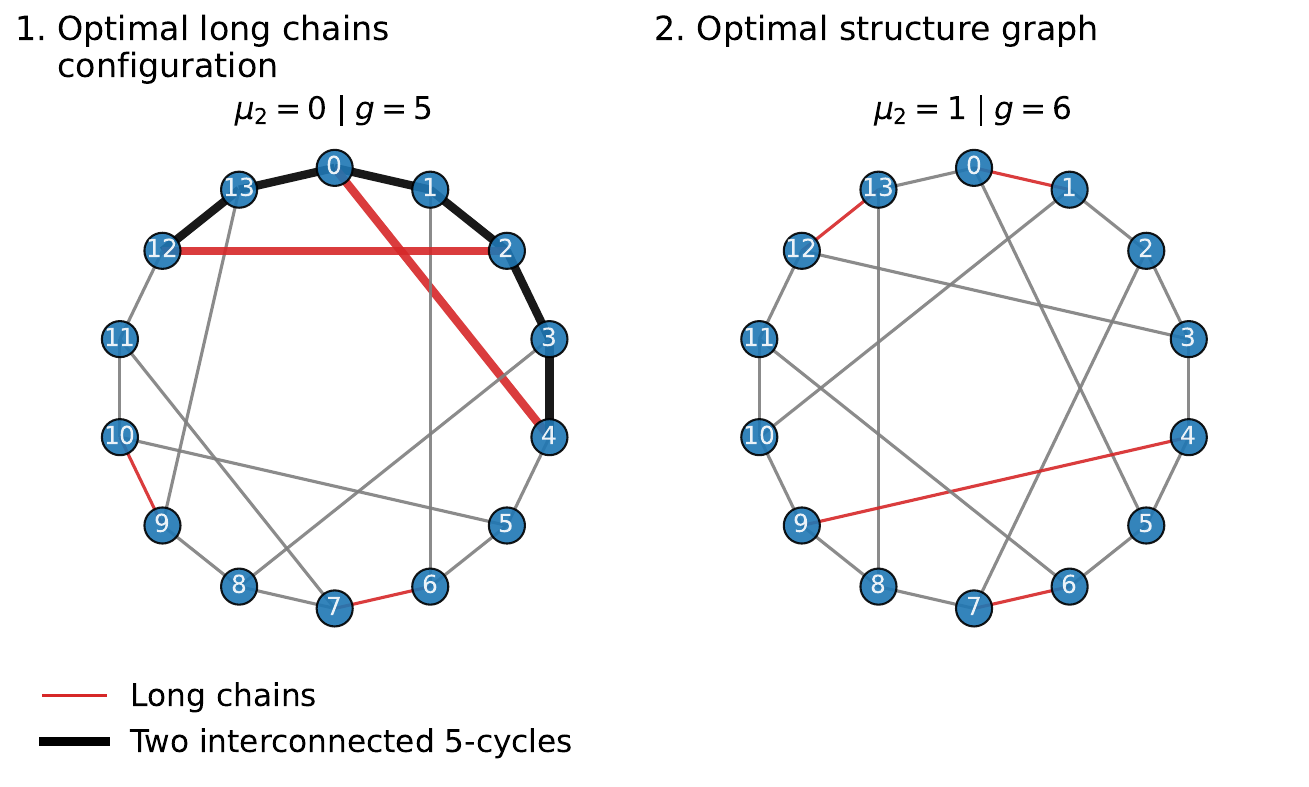}
    \caption{Two candidate marked structures for $r = 8$, $\lambda = 4$: although the right graph has higher girth, the left graph is more reliable near zero due to better separation among long chains and the presence of two interconnected 5-cycles.}
            \label{fig:long_chains}
\end{figure}

\paragraph{Integer linear programming.} To systematically identify optimal long-chain configurations, we formulate the problem
as an integer linear program (ILP). For each chain $e_i$ of the structure graph $G$,
we introduce a binary variable $x_i \in \{0,1\}$ indicating whether $e_i$ is selected
as a long chain. With this representation, all structural quantities minimized in
Algorithm~\ref{alg:most_reliable} can be expressed using linear constraints and
objective functions.

For example, the separation coefficient $\mu_d$, which counts the number of pairs
of long chains at graph distance $d$, can be modeled by introducing auxiliary binary
variables $x_{ij}$ for each pair of edges $(e_i,e_j)$ at distance $d$. The variable
$x_{ij}$ indicates whether both $e_i$ and $e_j$ are long chains and is enforced by
the linear constraint
\[
x_{ij} \ge x_i + x_j - 1.
\]
The separation coefficient $\mu_d$ is then given by
\[
\mu_d = \sum_{(i,j)\in \mathcal{P}_d} x_{ij},
\]
where $\mathcal{P}_d$ denotes the set of unordered pairs of chains whose distance
in $G$ is exactly $d$.

The resulting ILP takes the form
\begin{equation}
\label{eq:ilp_separation}
\begin{aligned}
\min \quad & \mu_d = \sum_{(i,j)\in \mathcal{P}_d} x_{ij} \\
\text{s.t.} \quad
& \sum_i x_i = \lambda, \\
& x_{ij} \ge x_i + x_j - 1 \quad \forall (i,j)\in \mathcal{P}_d, \\
& x_i \in \{0,1\}, \quad x_{ij} \in \{0,1\}.
\end{aligned}
\end{equation}

The same construction applies to the reliability obstruction parameters from
Lemma~\ref{lem:properties_sparse}, allowing the corresponding optimization problems
to be solved sequentially or lexicographically using ILP.

The full list of the optimal long-chain configurations is presented in the cubic database~\footnote{\url{https://github.com/RotemBrand/cubics_db}}.

\subsection{Empirical Analysis and Numerical Results}
\paragraph{Reliability percentile plots.} To further investigate the relationship between girth and reliability for larger values of $p$, we introduce \emph{reliability percentile plots}, where each line corresponds to a cubic graph. For each $p$ we sort the graphs according to their reliability. The rank is scaled to the range $[0,1]$, with $0$ indicating the most reliable graph and $1$ the least reliable at each value of~$p$. Graphs are colored according to the girth of their underlying structure graph.

The main goal of these plots is to assess whether graphs with higher girth tend to be more reliable, not only for small values of~$p$, as suggested by our theoretical analysis, but across the entire range of~$p$. The results provide strong empirical evidence supporting this hypothesis: higher-girth graphs consistently achieve higher reliability ranks across a wide range of~$p$.

To ensure that the ranking remains numerically meaningful, we restrict the upper bound of~$p$ to $0.8$. This is because the last $n - 2$ coefficients of the unreliability polynomial are identical across all graphs, causing their reliability values to converge for large~$p$. As a result, distinguishing between graphs becomes numerically unstable in that regime.

\begin{figure}
    \centering
    \includegraphics[width=0.75\linewidth]{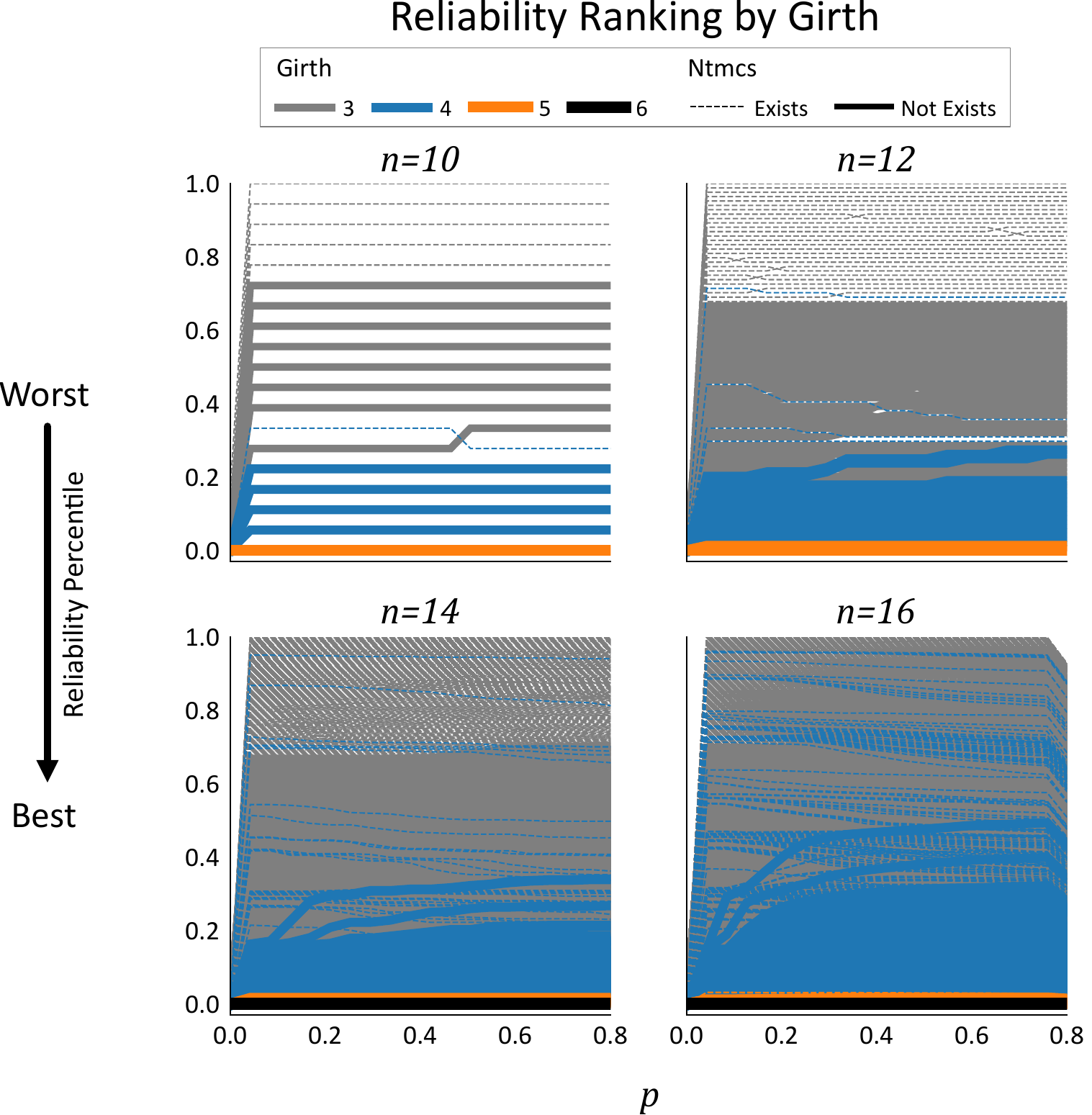}
    \caption{Cubic graphs with higher girth are consistently more reliable across a range of edge failure probabilities. The $y$-axis shows the normalized rank of the reliability function $U_G(p)$, where lower values indicate higher reliability. Each subplot displays all cubic graphs of a fixed size, colored by their girth.}
    \label{fig:graphs_compere}
\end{figure}

\section{Discussion}
This work provides a combinatorial framework for identifying the most reliable graphs with a fixed number of vertices and edges under independent edge failures. Our main contribution is the characterization of high-girth regular graphs as the most reliable in the regime of small edge failure probability $p$. This result identifies a unique candidate for the uniformly most reliable graph across several redundancy levels, in contrast to previous studies that typically focus on a single redundancy level.

Moreover, our results show that each cubic graph that is most reliable near $p=0$ serves as the structural core of the most reliable graph with the same level of redundancy and a sufficiently large number of vertices. In this way, high-reliability cubic graphs give rise to infinite families of highly reliable sparse graphs through appropriate chain constructions. At the same time, our framework provides a systematic criterion for ruling out uniform optimality: when the resulting marked structure is not tree-balanced, no uniformly most reliable graph exists for the corresponding redundancy level. These observations emphasize the central role of high-reliability cubic graphs, both as generators of reliable network families and as a basis for excluding uniform optimality in a principled manner.

By applying this framework, we construct a database of cubic graphs that are most reliable near $p=0$ for almost all redundancy levels up to $r=20$, together with the corresponding long-chain configurations. For each redundancy level in this range, the database identifies the full set of structural candidates from which uniformly most reliable graphs can arise for a sufficiently large number of vertices.

The primary limitation of our approach is its reliance on computational enumeration, which becomes increasingly intractable as the number of redundant edges grows. In addition, several of our results are asymptotic in nature. For instance, while we identify structural properties of uniformly most reliable graphs for sufficiently large numbers of vertices, we do not establish explicit bounds on the minimal graph size for which these properties hold. Future work could focus on further narrowing the enumeration search space and on sharpening the asymptotic bounds obtained in this study.

Several open questions remain. The most significant is whether the cubic graphs identified in our database are indeed uniformly most reliable, or whether counterexamples arise beyond the range considered here. As a concrete example, it would be interesting to determine whether regular graphs of maximal girth also maximize the number of spanning trees. Another natural direction is to estimate the number of cubic graphs with maximal girth as a function of the number of vertices, thereby providing explicit bounds on the complexity of enumeration-based approaches. In addition, constructing explicit families of high-girth cubic graphs and proving the absence of non-trivial minimal cut sets could further tighten existing lower bounds on the girth of the most reliable graphs. Finally, developing analytic methods for generating optimal long-chain configurations remains an important open problem, as it would complete the construction of highly reliable sparse networks without reliance on exhaustive computation.


\bibliographystyle{plain} 
\bibliography{bib} 

\end{document}